\documentclass{amsart}

\usepackage{latexsym}

\usepackage{amsfonts,amssymb,euscript,epsfig,color}

\usepackage[all]{xy}

\newtheorem{theo}{Theorem}[section]
\newtheorem{defi}[theo]{Definition}

\newtheorem{lem}[theo]{Lemma}

\newtheorem{prop}[theo]{Proposition}

\newtheorem{rem}[theo]{Remark}
\newtheorem{coro}[theo]{Corollary}

\newtheorem{exam}[theo]{Example}
\newtheorem{assu}[theo]{Assumption}

\newcommand{\lgot}{\ensuremath{\mathfrak{l}}}
\newcommand{\cgot}{\ensuremath{\mathfrak{c}}}

\newcommand{\kgot}{\ensuremath{\mathfrak{k}}}
\newcommand{\hgot}{\ensuremath{\mathfrak{h}}}
\newcommand{\ggot}{\ensuremath{\mathfrak{g}}}

\newcommand{\tgot}{\ensuremath{\mathfrak{t}}}
\newcommand{\pgot}{\ensuremath{\mathfrak{p}}}

\newcommand{\Rgot}{\ensuremath{\mathfrak{R}}}

\newcommand{\Acal}{\ensuremath{\mathcal{A}}}
\newcommand{\Bcal}{\ensuremath{\mathcal{B}}}
\newcommand{\Ccal}{\ensuremath{\mathcal{C}}}
\newcommand{\Dcal}{\ensuremath{\mathcal{D}}}
\newcommand{\Ecal}{\ensuremath{\mathcal{E}}}

\newcommand{\Kcal}{\ensuremath{\mathcal{K}}}
\newcommand{\Lcal}{\ensuremath{\mathcal{L}}}

\newcommand{\Ncal}{\ensuremath{\mathcal{N}}}
\newcommand{\Ocal}{\ensuremath{\mathcal{O}}}

\newcommand{\Qcal}{\ensuremath{\mathcal{Q}}}

\newcommand{\Ucal}{\ensuremath{\mathcal{U}}}
\newcommand{\Vcal}{\ensuremath{\mathcal{V}}}

\newcommand{\Xcal}{\ensuremath{\mathcal{X}}}

\newcommand{\RR}{\ensuremath{\mathrm{RR}}}

\newcommand{\Nbb}{\ensuremath{\mathbb{N}}}
\newcommand{\Cbb}{\ensuremath{\mathbb{C}}}
\newcommand{\Rbb}{\ensuremath{\mathbb{R}}}
\newcommand{\Zbb}{\ensuremath{\mathbb{Z}}}

\newcommand{\Ko}{\ensuremath{{\mathbf K}^0}}

\newcommand{\Chol}{\ensuremath{\Ccal_{\rm hol}}}
\newcommand{\chol}{\ensuremath{\Ccal^{\rho}_{\rm hol}}} 
\newcommand{\Ghol}{\ensuremath{\widehat{G}_{\rm hol}}}
\newcommand{\Khol}{\ensuremath{\widehat{K}_{\rm hol}}}

\newcommand{\Cholp}{\ensuremath{\Ccal'_{\rm hol}}}
\newcommand{\cholp}{\ensuremath{{\Ccal'}^{\rho}_{\rm hol}}} 
\newcommand{\Gholp}{\ensuremath{\widehat{G}'_{\rm hol}}}
\newcommand{\Kholp}{\ensuremath{\widehat{K}'_{\rm hol}}}

\newcommand{\Rforc}{\ensuremath{R^{-\infty}_{tc}}}
\newcommand{\Char}{\ensuremath{\hbox{\rm Char}}}

\newcommand{\As}{{\ensuremath{\rm As}}}

\newcommand{\Cr}{\ensuremath{\hbox{\rm Cr}}}

\newcommand{\croc}{\ensuremath{\hookrightarrow}}
\newcommand{\T}{\ensuremath{\hbox{\bf T}}}
\newcommand{\qfor}{\ensuremath{\mathcal{Q}^{-\infty}}}
\newcommand{\Rfor}{\ensuremath{R^{-\infty}}}

\newcommand{\indice}{\mathrm{Index}}

\def \what {\widehat}
\def \wtde {\widetilde}

\newcommand{\Thom}{\operatorname{Thom}}

\def \clif {\mathbf{c}}

\def \K {\mathbf{k}}

\def \clif {{\bf c}}

\def \Ad {{\rm Ad}}

\begin{document}

\title[ \hbox{\normalsize [Q,R]=0} in the non-compact setting]{Quantization commutes with reduction in the non-compact setting: the case of holomorphic discrete series}

\author{Paul-Emile  PARADAN}

\address{Institut de Math\'ematiques et de Mod\'elisation de Montpellier (I3M), 
Universit\'e Montpellier 2} 

\email{Paul-Emile.Paradan@math.univ-montp2.fr}

\date{January 2012}


\begin{abstract}
In this paper we show that the multiplicities of holomorphic discrete series representations relatively to reductive subgroups 
satisfy the credo ``Quantization commutes with reduction''.
\end{abstract}


\maketitle

{\def\thefootnote{\relax}
\addtocounter{footnote}{-1}
}

{\small
\tableofcontents}

\section{Introduction} 

The orbit method, introduced by Kirillov in the 60's, proposes a correspondence between the irreducible unitary representations of a Lie group G and its orbits in the coadjoint representation : the representation 	$\pi^G_\Ocal$ should be the geometric quantization of the Hamiltonian action of $G$ on the coadjoint orbit $\Ocal\subset\ggot^*$. The important feature of this correspondence is the functoriality relatively to inclusion $H\croc G$ of closed subgroups. It means that if we start with representations $\pi^G_{\Ocal}$ and $\pi^H_{\Ocal'}$ attached to the coadjoint orbits $\Ocal\subset\ggot^*$ and $\Ocal'\subset\hgot^*$, one expects that the multiplicity of $\pi^H_{\Ocal'}$ in the restriction $\pi^G_{\Ocal}\vert_H$ can be computed in terms of the space 
\begin{equation}\label{eq:reduced-space-intro}
\Ocal\cap\pi_{\hgot,\ggot}^{-1}(\Ocal')/H,
\end{equation} 
where $\pi_{\hgot,\ggot}:\ggot^*\to\hgot^*$ denotes the canonical projection. The symplectic geometers recognise that  (\ref{eq:reduced-space-intro}) is a symplectic reduced space in the sense of Marsden-Weinstein, since  $\pi_{\hgot,\ggot}:\Ocal\to\hgot^*$ is the moment map relative to the Hamiltonian action of $H$ on $\Ocal$. Let us give some examples where this theory is known to be valid.

For simply connected nilpotent Lie groups, Kirillov \cite{Kirillov62} described the correspondence $\Ocal\mapsto \pi^G_{\Ocal}$, and Corwin-Greenleaf \cite{Corwin-Greenleaf} proves its functoriality relatively to subgroup : the multiplicity appearing in the direct integral decomposition of 
$\pi^G_{\Ocal}\vert_H$ is the cardinal of the reduced space (\ref{eq:reduced-space-intro}). 

For compact Lie group, G. Heckman \cite{Heckman82} proved that the multiplicity was asymptotically given by the volume of the reduced space (\ref{eq:reduced-space-intro}). Just after Guillemin and Sternberg \cite{Guillemin-Sternberg82} replaced this functoriality principle in a more geometric framework and proposed another version of this rule for a good quantization process: \emph{the quantization should commute with the reduction}. This means that if $\Qcal_H(M)$ is the geometric quantization of an Hamiltonian action of a compact Lie group $H$ on a symplectic manifold $M$, the multiplicity of 
the representation $\pi^H_{\Ocal'}$ in $\Qcal_H(M)$ should be the (dimension of the) quantization of the reduced space 
$(\Phi^H_M)^{-1}(\Ocal')/H$. Here $\Phi^H_M: M\to \hgot^*$ denotes the moment map. 

A good quantization process for compact Lie group action on compact symplectic manifolds turns to be the equivariant index of a Dolbeault-Dirac operator \cite{Sjamaar96,Vergne-Bourbaki}. In the late 90's, Meinrenken and Meinrenken-Sjamaar proved that the principle of Guillemin-Sternberg works in this setting \cite{Meinrenken98,Meinrenken-Sjamaar}. Afterwards, this quantization procedure was extended to non-compact manifolds with a proper moment map by Ma-Zhang and the author \cite{pep-formal,Ma-Zhang,pep-formal-2}. See also the recent work of Duflo-Vergne on the multiplicities of the tempered representations relatively to compact subgroups \cite{Duflo-Vergne2011}.

The purpose of this article is to show that the \emph{quantization commutes with reduction} principle holds in a case where the group of symmetry is a real reductive Lie group.  Loosely speaking, we prove that if $\pi^G_{\Ocal}$ and $\pi^H_{\Ocal'}$ are holomorphic discrete series representations of real reductive Lie groups $H\subset G$, the multiplicity of  $\pi^H_{\Ocal'}$ in the restriction $\pi^G_{\Ocal}\vert_H$ is equal to the quantization of the reduced space (\ref{eq:reduced-space-intro}).

We turn now to a description of the contents of the various sections, highlighting the main features.

In Section \ref{sec:hamiltonian-reductive}, we clarify previous work of Weinstein \cite{Weinstein01} and Duflo-Vargas 
\cite{Duflo-Vargas2007,Duflo-Vargas2010}  concerning the Hamiltonian action of a connected reductive real Lie group $G$ on a symplectic manifold $M$.  The main point is that if the action of $G$ on $M$ is \emph{proper} and the moment map $\Phi^G_M:M\to\ggot^*$ relative to this action is \emph{proper}, then the image of $\Phi^G_M$ is contained in the strongly elliptic 
open subset $\ggot^*_{se}$ and the manifold has a decomposition
\begin{equation}\label{eq:decomposition-intro}
M=G\times_K Y.
\end{equation}
Here $K$ is a maximal compact subgroup of  $G$, and $Y$ is a closed $K$-invariant symplectic sub-manifold of $M$. 
Thanks to (\ref{eq:decomposition-intro}), we remark that the reduced space $(\Phi^G_M)^{-1}(\Ocal)/G$ is connected 
for any coadjoint orbit $\Ocal\subset\ggot^*$ : this is a notable difference with the nilpotent case where the reduced space 
(\ref{eq:reduced-space-intro}) can be disconnected.

The decomposition (\ref{eq:decomposition-intro}) will be the main ingredient of this paper to prove some \emph{quantization commutes with reduction} phenomenon. Note that P. Hochs already used this idea when  the manifold $Y$ is compact to get a \emph{quantization commutes with reduction} theorem  in the setting of KK-theory \cite{Hochs09}. Hochs was working on some induction process,  while we will use (\ref{eq:decomposition-intro}) to prove some functoriality relatively to a restriction procedure.

In this context, it is natural to look at the induced action of a reductive subgroup $G'\subset G$ on $M$, and we know then that we have 
another decomposition  $M=G'\times_{K'} Y'$ if the moment map $\Phi^{G'}_M$ is proper. In Section \ref{sec:criterion}, we give a 
criterion that insures the properness of $\Phi^{G'}_M$.

In Section \ref{sec:QR-hol}, we turn onto a closed study of the holomorphic discrete series representations of a reductive Lie group $G$. 
Recall that the parametrization of these representations depends on the choice of an element $z$ in the center of the Lie algebra of $K$ 
such that the adjoint map $\mathrm{ad}(z)$ defines a complex structure on $\ggot/\kgot$. Let $T$ be a maximal torus in $K$, with Lie algebra $\tgot$. The existence of element $z$ forces $\tgot$ to be a Cartan sub-algebra of $\ggot$, and it defines a closed cone $\chol(z)\subset \tgot^*$. If $\wedge^*\subset \tgot^*$ is the weight lattice, we consider the subset 
$$
\Ghol(z):=\chol(z)\cap \wedge^*_+
$$
where $\wedge^*_+$ is the set of dominant weights. The work of Harish-Chandra tells us that that we can attach an holomorphic discrete series representation $V^G_\lambda$ to any $\lambda\in \Ghol(z)$.  

In Section \ref{sec:QR-hol}, we look at formal quantization procedures attached to the Hamiltonian action of $G$ on a symplectic manifold $M$. We suppose that the properness assumptions are satisfied and that the image of the moment map $\Phi^G_M$ is contained in 
$G\cdot\chol(z)\subset\ggot^*_{se}$. Let us briefly recall the definition. We define the formal geometric quantization of 
the $G$-action on $M$ as the following formal sum 
\begin{equation}\label{eq:formal-G-intro}
\qfor_G(M):=\sum_{\lambda\in \Ghol(z)} \Qcal(M_{\lambda,G})\ V^G_\lambda,
\end{equation}
where $\Qcal(M_{\lambda,G})\in\Zbb$ is the quantization of the \emph{compact} symplectic reduced space $M_{\lambda,G}:=(\Phi^G_M)^{-1}(G\cdot\lambda)/G$.

Since the moment map $\Phi^K_M$ is proper, we can also define the formal geometric quantization of the $K$-action on $M$ as 
 \begin{equation}\label{eq:formal-K-intro}
\qfor_K(M):=\sum_{\mu\in \wedge^*_+} \Qcal(M_{\mu,K})\ V^K_\mu,
\end{equation}
where $M_{\mu,K}:=(\Phi^K_M)^{-1}(K\cdot\mu)/K$, and $V^K_\mu$ denotes the irreducible representation of $K$ with highest weight $\mu$. The formal quantization procedure $\qfor_K$, together with its functorial properties, 
has been studied by Ma-Zhang and the author in \cite{pep-formal,Ma-Zhang,pep-formal-2}.

Let $\Rfor(G,z)$ be the $\Zbb$-module formed by the infinite sums $\sum_{\lambda\in \Ghol(z)} m_\lambda \ V^{G}_\lambda$ with  
$m_\lambda\in\Zbb$. We consider also the $\Zbb$-module $\Rfor(K)$ formed by the infinite sums $\sum_{\mu\in \wedge^*_+} n_\mu \ V^{K}_\mu$ with $n_\mu\in\Zbb$. The following basic result will be an important tool in our paper (see Lemma \ref{lem:restriction-G-K}).

\medskip 
{\bf Lemma A}\ \emph{We have an injective restriction morphism $\mathbf{r}_{K,G} : \Rfor(G,z)\to\Rfor(K)$.}

\medskip

We can state one of our main result

\medskip

{\bf Theorem B}\ \emph{If Assumptions\footnote{See Assumptions \ref{assumption}.} {\bf A1} or {\bf A2} are satisfied, we have the following relation
$$
\mathbf{r}_{K,G}\Big(\qfor_G(M)\Big)=\qfor_K(M).
$$ 
}

\medskip

The main tool for proving Theorem B, is the following relation 
$$
 \qfor_K(M)=\qfor_K(Y) \otimes S^\bullet(\pgot)
 $$
where $\qfor_K(Y)$ is the formal geometric quantization of the slice $Y$ and $S^\bullet(\pgot)$ is the symmetric algebra of the 
complex $K$-module $\pgot:=(\ggot/\kgot,\mathrm{ad}(z))$. 

\medskip

Consider now a connected reductive subgroup $G'\subset G$ such that its Lie algebra $\ggot'$ contains the element $z$. Let $\Phi^{G'}_{G\cdot\lambda}$ be the moment map relative to the Hamiltonian action of $G'$ on the coadjoint orbit $G\cdot\lambda,\,\lambda\in\Ghol(z)$. It is not difficult to see that $\Phi^{G'}_{G\cdot\lambda}$ is a proper map. Thanks to the work of  
T. Kobayashi \cite{Toshi-inventiones94} and Duflo-Vargas \cite{Duflo-Vargas2007}, we know that the representation $V^G_\lambda$ admits an admissible restriction to $G'$. It means that the restriction $V^G_\lambda\vert_{G'}$ is a discrete sum formed by holomorphic discrete series representations $V^{G'}_\mu,\, \mu\in \Ghol'(z)$.

We can now state the major result of this paper (see Theorem \ref{theo:QR=0-non-compact}).

\medskip

{\bf  Theorem C}\ \emph{ Let $\lambda\in \Ghol(z)$. We have the following relation
$$
V^G_\lambda\vert_{G'}=\qfor_{G'}(G\cdot\lambda).
$$
It means that for any $\mu\in \Gholp(z)$, the multiplicity of the representation $V^{G'}_\mu$ in 
the restriction $V^G_\lambda\vert_{G'}$ is equal to the geometric quantization
$\Qcal\left((G\cdot\lambda)_{\mu,G'}\right)\in \Zbb$ 
of the (compact) reduced space $(G\cdot\lambda)_{\mu,G'}$.
}

\medskip

In Section \ref{sec:formal-G}, we prove that the formal quantization process $\qfor_G$ is functorial relatively to reductive subgroups.

\medskip

{\bf Theorem D}\ \emph{Suppose that Assumption {\bf A2} holds. Then we have 
$$
\mathbf{r}_{G',G}\left(\qfor_G(M)\right)=\qfor_{G'}(M).
$$
where $\mathbf{r}_{G',G}$ is a restriction morphism.
}

\medskip

In \cite{Jakobsen-Vergne}, Jakobsen-Vergne proposed another formula for the the multiplicity of the representation $V^{G'}_\mu$ in 
the restriction $V^G_\lambda\vert_{G'}$. In Section \ref{subsec:quantization-Y}, we explain how to recover their result from Theorem B.

\medskip

Section \ref{sec:transversally} is devoted to the proofs of the main results of this paper.  We use here previous work of the author 
on localization techniques in the the setting of transversally elliptic operators.

\bigskip

{\bf Notations.}\ In this paper, $G$ will denoted a connected real reductive Lie group. We take here the convention of Knapp \cite{Knapp-book}. 
We have a Cartan involution $\Theta$ on $G$, such that the fixed point set $K:=G^\Theta$ is a connected maximal compact subgroup. We have Cartan decompositions : at the level of Lie algebras $\ggot=\kgot\oplus \pgot$ and at the level of the group $G\simeq K\times \exp(\pgot)$. 
We denote by $b$ a $G$-invariant non-degenerate bilinear form on $\ggot$ that defines a $K$-invariant scalar product $(X,Y):=-b(X,\Theta(Y))$.

When $V$ and $V'$ are two representations of a group $H$, the multiplicity of $V$ in $V'$ will  be denoted $[V:V']$.

\section{Hamiltonian actions of real  reductive Lie groups}\label{sec:hamiltonian-reductive}

This section is mainly a synthesis of previous work by Weinstein \cite{Weinstein01}, Duflo-Vargas 
\cite{Duflo-Vargas2007,Duflo-Vargas2010} and Hochs \cite{Hochs09}, except the criterion that we obtain in Section \ref{sec:criterion}.

Let $G$ be a connected real reductive Lie group.  We consider an Hamiltonian action of $G$ on a connected symplectic manifold $(M,\Omega_M)$. The corresponding moment map $\Phi_M^G:M\to\ggot^*$ is defined (modulo a constant) by the relations
\begin{equation}\label{eq:kostant=rel}
\iota(X_M)\Omega_M=-d\langle\Phi_M^G,X\rangle,\quad \forall X\in\ggot,
\end{equation}
where $X_M(m):=\frac{d}{ds}e^{-s X}\cdot m\vert_{s=0}$ is the vector field generated by $X\in\ggot$.

Let $\ggot=\kgot\oplus \pgot$ be a Cartan decomposition. Let $K\subset G$ be the maximal compact subgroup with Lie algebra $\kgot$. 
Thus we have a decomposition 
$$
\Phi_M^G=\Phi_M^K \oplus \Phi_M^\pgot
$$
where $\Phi_M^K:M\to\kgot^*$ is a moment map relative to the action of $K$ on $(M,\Omega_M)$, and $\Phi_M^\pgot:M\to\pgot^*$ is 
$K$-equivariant. 

The $K$-invariant scalar product  $\|X\|^2= -b(X,\Theta(X))$ on $\ggot$ induces an identification $\xi\mapsto \tilde{\xi},\ \ggot^*\simeq \ggot$ defined by $(\tilde{\xi},X):=\langle\xi,X\rangle$ for $\xi\in\ggot^*$ and $X\in\ggot$. We still denote $\|\xi\|^2:=\|\tilde{\xi}\|^2$ the corresponding scalar product on $\ggot^*$.

Let $\kappa^G,\kappa^K$ and $\kappa^\pgot$ be respectively the Hamiltonian vectors fields of the $K$-invariant functions
$\frac{-1}{2}\|\Phi^G_M\|^2$, $\frac{-1}{2}\|\Phi^K_M\|^2$, and $\frac{-1}{2}\|\Phi^\pgot_M\|^2$. 
The relations (\ref{eq:kostant=rel}) give that 
\begin{equation}\label{eq:kirwan=vector}
\kappa^{-}(m)=\left[\widetilde{\Phi^{-}_M(m)}\right]_M(m),\quad \forall m\in M,
\end{equation}
for $-\in\{G,K,\pgot\}$.

\subsection{Proper actions}

In this section we suppose $\mathbf{C1 :}$ the action of $G$ on $M$ is proper\footnote{For any compact subset $A$ of $M$ the subset $\{g\in G \,\vert\, g\cdot A\cap A\neq \emptyset\}$ is compact.}. We have then the fundamental fact.

\begin{lem}
$\bullet$ The map $\Phi_M^\pgot:M\to\pgot^*$ is  a $K$-equivariant submersion, so for any $a\in\pgot^*$, the fiber
$Y_a:=(\Phi_M^\pgot)^{-1}(a)$ is either empty or a sub-manifold of $M$.

$\bullet$ The set of critical points of $\|\Phi_M^\pgot\|^2:M\to\Rbb$ is $Y_0:=(\Phi_M^\pgot)^{-1}(0)$.

\end{lem} 

\begin{proof} Let us prove the first point. Let $m\in M$. Since the tangent map $ \T \Phi_M^\pgot(m):\T_m M\to \pgot^*$ satisfies
\begin{equation}\label{rel=1}
\langle \T \Phi_M^\pgot(m),X\rangle=- \iota(X_M)\Omega_M\vert_m, \quad \forall X\in\pgot,
\end{equation} 
the orthogonal of the image of $\T \Phi^\pgot_M(m)$ is equal to $\pgot_m:=\{X\in\pgot\ |\  X_M(m)=0\}$. 
As the action of $G$ on $M$ is proper,  the stabilizer subgroup $G_m$ is compact. This forces 
$\pgot_m= {\rm Lie}(G_m)\cap\pgot$ to be reduced to $\{0\}$. Thus $\T \Phi_M^\pgot(m)$ is onto and the first point is proved. 
Let $m\in M$ be a critical point of $\|\Phi_M^\pgot\|^2$. The Hamiltonian vector field $\kappa^{\pgot}$ vanishes at $m$, and (\ref{eq:kirwan=vector}) tells you that $\widetilde{\Phi_M^\pgot(m)}\in \pgot_m=\{0\}$. The second point is proved.
\end{proof}

\bigskip

For the remaining part of this section, we consider the $K$-invariant sub-manifold 
$$
Y:=Y_0\subset M
$$
that we suppose non empty. Let us consider the restriction $\Omega_{Y}$ of the symplectic structure $\Omega_M$ to $Y$. For $y\in Y$, let $\pgot\cdot y=\{ X_M(y),\ X\in \pgot\}\subset \T_y M$. The tangent space $\T_y Y$ is by definition the kernel of $\T \Phi_\pgot(y)$. Relations (\ref{rel=1}) show that  
\begin{equation}\label{rel=3}
\T_y Y= (\pgot\cdot y)^\perp
\end{equation} 
where the orthogonal is taken relatively to the symplectic form. Hence the kernel of $\Omega_{Y}\vert_y$ is equal to $(\pgot\cdot y)^\perp\cap \pgot\cdot y$. For $X,X'\in\pgot$ and $y\in Y$, we have 
\begin{eqnarray*}
\Omega_M(X_M(y),X'_M(y))&=& \langle \Phi_M^G(y),[X,X']\rangle\\
&=& \langle [\Phi_M^K(y),X],X'\rangle.
\end{eqnarray*}

Hence $(\pgot\cdot y)^\perp\cap \pgot\cdot y\simeq \ggot_\xi\cap \pgot$ for $\xi= \Phi_M^K(y)$. Note that for $\xi\in\kgot$, we have 
$\ggot_\xi=\ggot_\xi\cap\kgot\oplus \ggot_\xi\cap\pgot$. We have then proved 

\begin{lem} \label{lem:non-degenerate}
Let $y\in Y$. The $2$-form $\Omega_{Y}\vert_y$ is non degenerated if and only $\ggot_\xi\subset\kgot$
for $\xi= \Phi_M^K(y)$.
\end{lem}

We have a canonical $G$-equivariant map $\pi: G\times_K Y \longrightarrow M$ that sends $[g,y]$ to $g\cdot y$. Following Weinstein 
\cite{Weinstein01}, we consider the $G$-invariant open subset 
\begin{equation}\label{eq:def-se}
\ggot^*_{se}=\{\xi\in\ggot^*\,\vert\, G_\xi \ \mathrm{is\  compact} \}
\end{equation}
of strongly elliptic elements. 

\begin{exam}
For the group $G=\mathrm{SL}_2(\Rbb)$, the set $\ggot^*_{se}$ is equal, through the trace,  to the cone
$\{X\in\mathfrak{sl}_2(\Rbb)\ \vert \det(X)>0\}$.
\end{exam}

Let $\tgot$ be the Lie algebra of a maximal torus in $K$. Weinstein proves that 
the open subset $\ggot^*_{se}$ is non-empty if and only if  $\tgot$ is a Cartan sub-algebra of $\ggot$ \cite{Weinstein01}. See the Table \ref{tab:strongly-elliptic}. We note that $\kgot^*_{se}:=\ggot^*_{se}\cap\kgot^*$ is equal to  
$\{\xi\in\kgot^*\,\vert\, G_\xi \subset K \}$ and that
\begin{equation}\label{eq:se}
\ggot^*_{se}={\rm Ad}^*(G)\cdot\kgot^*_{se}.
\end{equation}

\begin{table}
   \caption{\label{tab:strongly-elliptic} Strongly elliptic set}
   \begin{center}
   \begin{tabular}{c|c}
      $G $               &   $\ggot^*_{se}\neq \emptyset$    \\
 \hline

 $\mathrm{GL}(n,\Cbb)$    & no               \\
 $\mathrm{O}(n,\Cbb)$    & no               \\
 $\mathrm{SL}(n,\Rbb)$    & $n=2$               \\
 $\mathrm{SO}_o(p,q)$    & $pq$ even  \\
 $\mathrm{Sp}(n,\Rbb)$    & yes               \\
$\mathrm{SO}^*(2n)$       & yes               \\
$\mathrm{U}(p,q)$            & yes                \\
$\mathrm{Sp}(p,q)$            & yes                \\
   \end{tabular}
   \end{center}
\end{table}

Let us consider the invariant open subsets $M_{se}= (\Phi^G_M)^{-1}(\ggot^*_{se})\subset M$ and $Y_{se}:= Y\cap M_{se}\subset Y$.

\medskip

\begin{lem}\label{lem:M=se} 

$\bullet$ The $2$-form $\Omega_Y$ is non-degenerated on $Y_{se}$. 

$\bullet$ The action of the group $K$ on $(Y_{se},\Omega_{Y_{se}})$ is Hamiltonian, with moment map $\Phi^K_{Y_{se}}$ that is equal to the restriction of $\Phi^G_M$ to $Y_{se}$.

$\bullet$ The map $\pi$ induces a $G$-equivariant diffeomorphism
$\pi_{se}: G\times_K Y_{se} \to M_{se}$.
\end{lem}

\begin{proof} The first point is a direct consequence of Lemma \ref{lem:non-degenerate}. The second point is immediate. Let us check the last point. 

Relation (\ref{eq:se}) shows that $\pi_{se}$ is onto. Let $[g_0,y_0],[g_1,y_1]$ such that $g_0\cdot y_0=g_1\cdot y_1$. Then by taking the image by the moment map, we get $\mathrm{Ad}^*(g_0)\xi_0=\mathrm{Ad}^*(g_1)\xi_1$ where the $\xi_k=\Phi^K_M(y_k)$ belong to 
$\kgot^*_{se}$. Let $h=g_1^{-1}g_0\in G$. We have $\mathrm{Ad}^*(h)\xi_0=\xi_1$, and $\mathrm{Ad}^*(\Theta(h))\xi_0=\xi_1$ by taking the Cartan involution. Finally $h^{-1}\Theta(h)\in G_{\xi_0}$. Since $G_{\xi_0}\subset K$, we get that $h\in K$, and finally that $[g_0,y_0]=[g_1,y_1]$ 
in $G\times_K Y_{se}$.
\end{proof}

\bigskip

Let us denote by $\Omega_{M_{se}}$ the restriction of the symplectic form $\Omega_M$ on the open subset $M_{se}$. We will now finish this section by giving a simple expression of the pull-back $\pi_{se}^*(\Omega_{M_{se}})\in \Acal^2(G\times_K  Y_{se})$.  

Let $\theta^G\in\Acal^1(G)\otimes\ggot$ be  the canonical connexion $1$-form relative to the $G$-action by right translations :
$\iota(X^r)\theta^G=X$, $\forall X\in \ggot$, where $X^r(g)=\frac{d}{dt} (g e^{tX})\vert_0$. Let $\theta^K\in\Acal^1(G)\otimes\kgot$, the composition of $\theta^G$ with the orthogonal projection $X\to X_\kgot$ from $\ggot$ to $\kgot$. We will use the  $G\times K$-invariant $1$-form on $G\times Y_{se}$ defined by $\langle \theta^K,\Phi^K_{Y_{se}}\rangle$.

Note that the space of differentials forms on $G\times_K Y_{se}$  admits a canonical identification with the space of $K$-basic differentials forms on $G\times Y_{se}$. 

\begin{prop}The $2$-form $\pi_{se}^*(\Omega_{M_{se}})$ is equal to the $K$-basic, $G$-invariant, $2$-form
$\Omega_{Y_{se}}- d\langle\Phi^K_{Y_{se}}, \theta^K\rangle$.
\end{prop}

\begin{proof} Let $\pi_1: G\times Y_{se}\to M_{se}$ the map that factorizes $\pi_{se}$. By $G$-invariance, we need only to show that 
$\pi_{1}^*(\Omega_{M_{se}})$ equals $ \Omega_{Y_{se}}- d\langle\Phi^K_{Y_{se}}, \theta^K\rangle$ 
at the point $(1,y)\in G\times  Y_{se}$. 

Let $(X',v'),(X,v)\in\ggot\times \T_y=\T_{(1,y)}(G\times  Y_{se})$.  We have
\begin{eqnarray*}
\lefteqn{\pi_{1}^*(\Omega_{M_{se}})\Big( (X',v'), (X,v)\Big)}\\
&=&\Omega_M( -X'_M(y) + v', -X_M(y) + v)\\
&=& \Omega_M(v',v) + \Omega_M(X'_M(y),X_M(y)) - \Omega_M( X'_M(y) , v) + \Omega_M( X_M(y) , v')\\
&=& \Omega_{Y_{se}}(v',v) + \underbrace{\langle\Phi^K_{Y_{se}}(y), [X',X]_\kgot\rangle}_{{\bf A}} + 
\underbrace{d \langle \Phi^K_{Y_{se}},X'_\kgot\rangle\vert_y(v)-
d \langle \Phi^K_{Y_{se}},X_\kgot\rangle\vert_y(v')}_{{\bf B}}.\\
&=& \Omega_{Y_{se}}(v',v)- d\langle\Phi^K_{Y_{se}}, \theta^K\rangle\Big( (X',v'), (X,v)\Big).
\end{eqnarray*}
The last equality is due to the fact that  ${\bf A}= - \langle \Phi^K_{Y_{se}}(y),d\theta^K\vert_1\rangle(X',X)$ since $d\theta^K((X')^r,(X)^r)= -[X',X]_\kgot$, and ${\bf B}=- \langle d\Phi^K_{Y_{se}}, \theta^K\rangle( (X',v'), (X,v))$.

\end{proof}

\subsection{Proper moment map}\label{sec:proper-moment-map}

In this section we study the Hamiltonian actions of a real reductive group $G$ on a symplectic manifold $(M,\Omega_M)$ that meet the following condition:
\begin{enumerate}
\item[{\bf C1}] The action of $G$ on $M$ is proper,

\item[{\bf C2}]  The moment map $\Phi_M^G:M\to\ggot^*$ is a proper map\footnote{For any compact subset $B\subset\ggot^*$ the fiber $\Phi_G^{-1}(B)$ is a compact subset of $M$.}.
\end{enumerate}

\medskip

The condition {\bf C2} imposes that the image of $\Phi_M^G$ is a closed subset of $\ggot^*$. Let $\tilde{A}$ be a compact subset of 
${\rm Image}(\Phi_G)$, and let $A=(\Phi_M^G)^{-1}(\tilde{A})$ be the corresponding compact subset of $M$. We see then that, $\forall g\in G$  
$$
g\cdot A\cap A\neq \emptyset \Longleftrightarrow g\cdot \tilde{A}\cap \tilde{A}\neq \emptyset.
$$
Condition {\bf C1} tell us then that $\{g\in G \,\vert\, g\cdot A\cap A\neq \emptyset\}$ is compact, hence $\{g\in G \,\vert\, g\cdot \tilde{A}\cap \tilde{A}\neq \emptyset\}$ is compact for any compact set $\tilde{A}$ in the image of $\Phi_M^G$. By taking $\tilde{A}$ equal to a point, we get the following 

\begin{lem}
Under the conditions {\bf C1} and {\bf C2}, the image of $\Phi_M^G$ is contained in the open subset 
$\ggot^*_{se}=\{\xi\in\ggot^*\,\vert\, G_\xi \ \mathrm{is\  compact} \}$ of strongly elliptic elements. 
In particular, the image of $\Phi_M^G$ does not contain $0\in\ggot^*$.
\end{lem}

The previous Lemma gives a strong condition on the reductive Lie group $G$ : it may acts in an Hamiltonian fashion on a symplectic manifold,  {\em properly} and with a {\em proper} moment map only if $\ggot^*_{se}\neq \emptyset$. So $G$ can not be for example $SL_n(\Rbb)$ for $n\geq 3$ or a complex reductive Lie group (see Table \ref{tab:strongly-elliptic}).

\medskip

If we use the last section we see that $M=M_{se}$. We summarize with the following

\begin{prop}\label{prop:M.Y}

$\bullet$ The set $Y$ is a $K$-invariant symplectic sub-manifold of $M$, with proper moment map 
$\Phi^K_{Y}$ equal to the restriction of $\Phi^G_M$ to $Y$.

$\bullet$ The manifold $G\times_K Y$ carries an induced symplectic structure 
$\Omega_{Y}-d\langle\Phi^K_{Y},\theta^K\rangle$. The corresponding moment map is $[g,y]\mapsto g\cdot\Phi^K_{Y}(y)$.

$\bullet$ The map $\pi: G\times_K Y\to M$ is a $G$-equivariant diffeomorphism of Hamiltonian  $G$-manifolds.

$\bullet$ The manifold $Y$ is connected.

\end{prop}

\begin{proof}
Thanks to the Cartan decomposition, the third point implies that $\pgot\times Y\simeq M$ and then the last point follows.
\end{proof}

\bigskip

Let $\tgot$ be the Lie algebra of a maximal torus $T$ in $K$. Note that $\ggot^*_{se}\neq\emptyset$ is equivalent to the fact that 
$\tgot$ is a Cartan sub-algebra of $\ggot$. Let $\ggot^*_{se}=\ggot^*_{se}\cap\kgot^*$ and $\tgot^*_{se}=\ggot^*_{se}\cap\tgot^*$. We have
$\ggot^*_{se}={\rm Ad}^*(G)\cdot\kgot^*_{se}={\rm Ad}^*(G)\cdot \tgot^*_{se}$.

Let $\wedge^*\subset\tgot^*$ be the weight lattice : $\alpha\in\wedge^*$ if $i\alpha$ is the 
differential of a character of $T$. Let $\Rgot\subset \wedge^*$ be the set of roots 
for the action of $T$ on $\ggot\otimes\Cbb$. We have $\Rgot=\Rgot_c\cup \Rgot_n$ 
where $\Rgot_c$ and $\Rgot_n$ are respectively the set of roots for the action of $T$ on 
$\kgot\otimes\Cbb$ and $\pgot\otimes\Cbb$. We fix a system of positive roots $\Rgot^+_c$ in 
$\Rgot_c$: let $\tgot^*_+\subset\tgot^*$ be the corresponding Weyl chamber. Let $W=W(K,T)$ be the Weyl group.
We have then 
\begin{eqnarray*}
\tgot^*_{se}= W\cdot \left(\tgot^*_{se}\cap  \tgot^*_+\right)
&=&W\cdot\{\xi\in \tgot^*_+\ \vert \ (\xi,\alpha)\neq 0,\ \forall \alpha \in\Rgot_n\}\\
&=&W\cdot \left(\Ccal_1\cup\cdots\cup\Ccal_N\right),
\end{eqnarray*}
where each $\Ccal_j$ is an open cone of the Weyl chamber. 

We recover the following result due to Weinstein \cite{Weinstein01}.

\begin{theo}\label{prop:Kirwan.M.Y} 
$\bullet$ The Kirwan set $\Delta_K(Y):={\rm Image}(\Phi_{Y}^K)\cap\tgot^*_+$ is a closed convex locally polyhedral subset contained in one cone
$\Ccal_j$.

$\bullet$ We have ${\rm Image}(\Phi_M^G)/ \Ad^*(G)\simeq \Delta_K(Y)$.
\end{theo}

\begin{proof} Since the moment map $\Phi^K_{Y}$ is proper and $Y$ is connected, the Convexity Theorem \cite{Atiyah82,Guillemin-Sternberg82.bis,Kirwan.84.bis,L-M-T-W} tells us that the Kirwan set $\Delta_K(Y)$  is a closed, convex, locally polyhedral, subset of the Weyl chamber. On the other hand, we know that the image of $\Phi^K_{Y}$ belongs to 
$\tgot^*_{se}$. Then $\Delta_K(Y)\subset \Ccal_1\cup\cdots\cup\Ccal_N$, but since $\Delta_K(Y)$ is convex we have 
$\Delta_K(Y)\subset \Ccal_j$ for a unique cone $\Ccal_j$. The last point is obvious since the isomorphism $\pi: G\times_K Y\to M$ satisfies $\Phi^G_M\circ \pi([g,y])=g\cdot\Phi^K_{Y}(y)$.
\end{proof}

We finish this section, with the following

\begin{theo}\label{theo:properGK} 
Let $(M,\Omega_M,\Phi_M^G)$ be an Hamiltonian $G$-manifold. 

$\bullet$ If the $G$-action on $M$ is {\em proper}, $\Phi_M^G$ is {\em proper} if and only if $\Phi_M^K$ is {\em proper}.

$\bullet$ Under the conditions {\bf C1} and {\bf C2}, we have
$$
\emptyset\neq \Cr(\|\Phi_M^G\|^2)=\Cr(\|\Phi_M^K\|^2)=\Cr(\|\Phi_Y^K\|^2)\subset Y.
$$ 
\end{theo}

\begin{proof} Let us prove the first point. As $\|\Phi_M^G\|\geq \|\Phi_M^K\|$ one implication trivially holds. Suppose now that 
$\Phi_M^G$ is proper. Thanks to Propositions \ref{prop:M.Y} and \ref{prop:Kirwan.M.Y}, we know that $M=G\times_K Y$ where $Y$ is a $K$-Hamiltonian manifold, with proper moment map $\Phi_Y^K$, and with Kirwan set $\Delta_K(Y)$ being a closed set in 
$\tgot^*_{se}$. Let $R>0$.  We consider 
\begin{enumerate}
\item[$\bullet$] $M_{\leq R}=\{m\in M\ \vert\ \|\Phi_M^K(m)\|^2\leq R\}$,
\item[$\bullet$]  $Y_{\leq R}=\{y\in Y\ \vert\ \|\Phi_Y^K(y)\|^2\leq R\}$ which is a compact subset of $Y$,
\item[$\bullet$]  $\Kcal=\Delta_K(Y)\cap\{\xi\in\tgot^*\ | \ \|\xi\|^2\leq R\}$ which is a compact subset of $\tgot^*_{se}$,
\item[$\bullet$]  $c(\Kcal)=\inf_{\stackrel{\xi\in \Kcal}{\alpha\in\Rgot_n}}\frac{|(\alpha,\xi)|^2}{2\|\xi\|}$ which is strictly positive.
\end{enumerate}

We have to show that $M_{\leq R}$ is a compact subset of $M$.  Take $m=[ke^X,y]$, with $k\in K$ and $X\in\pgot$. Since 
$\Phi_M^G(m)= ke^X\cdot\Phi_Y^K(y)$, we have
\begin{eqnarray*}
\|\Phi_M^K(m)\|^2&\geq &-b(\Phi_M^G(m),\Phi_M^G(m))= \|\Phi_Y^K(y)\|^2\\
\|\Phi_M^K(m)\|^2&=& \| \left[e^X\cdot\Phi_Y^K(y)\right]_{\kgot^*}\|^2.
\end{eqnarray*}

Hence if  $m=[ke^X,y]\in M_{\leq R}$, we have $y\in Y_{\leq R}$ and then $\Phi_Y^K(y)=k_o\cdot\xi$ for some $k_o\in K$ and 
$\xi_o\in \Kcal$. Then we have, for $X'=k_o^{-1}\cdot X\in \pgot$,
\begin{eqnarray*}
\|\Phi_M^K(m)\|=\| \left[e^{X'}\cdot\xi_o\right]_{\kgot^*}\|&\geq&  \frac{1}{\|\xi_o\|}(e^{X'}\cdot\xi_o,\xi_o)\\
&=&  \frac{1}{\|\xi_o\|}\sum_{n\in\Nbb} \frac{1}{2n!}\|\mathrm{ad}^*(X')^{n}\xi_o\|^2\\
&\geq&  \frac{1}{2\|\xi_o\|}\|\mathrm{ad}^*(X')\xi_o\|^2\geq c(\Kcal)\|X\|^2.
\end{eqnarray*}
Thus if $m=[ke^X,y]\in M_{\leq R}$, the vector $X$ is bounded and $y$ belongs to the compact subset $Y_{\leq R}$. This proves that $M_{\leq R}$ is compact.

Let us concentrate to the last point. First we note that since the map $\|\Phi_M^G\|^2: M\to \Rbb$ is proper, its infimum is reached, and so 
$\Cr(\|\Phi_M^G\|^2)\neq \emptyset$. Let $-\in\{G,K\}$. Thanks to (\ref{eq:kirwan=vector}), we know that 
$$
m\in\Cr(\|\Phi_M^{-}\|^2)\Longleftrightarrow \kappa^{-}(m)=0\Longleftrightarrow \widetilde{\Phi_M^{-}(m)}\in\ggot_{m}.
$$
Since $\ggot_m\subset \ggot_{\xi}$ with $\Phi_M^G(m)=\xi=\xi_\kgot\oplus\xi_\pgot$, we have $m\in\Cr(\|\Phi_M^{-}\|^2)$ only if 
$[\tilde{\xi_\pgot},\tilde{\xi}]=0$. Since $\xi$ is strongly elliptic the last condition imposes that $\xi_\pgot=0$. We have proved that 
$\Cr(\|\Phi_M^K\|^2)$ and $\Cr(\|\Phi_M^G\|^2)$ are both contained in $\{\Phi_M^\pgot=0\}=Y$. We have $\kappa^G=\kappa^K+\kappa^\pgot$ and 
the vector field $\kappa^\pgot$ vanishes on $Y$. Finally we see that
\begin{eqnarray*}
\Cr(\|\Phi_M^G\|^2)=\Cr(\|\Phi_M^K\|^2)&=&\{y\in Y\ \vert\ [\widetilde{\Phi_M^K(y)}]_M(y)=0\}\\
&=&\Cr(\|\Phi_Y^K\|^2).
\end{eqnarray*}
The last equality is due to the fact that $\Phi_Y^K$ is the restriction of $\Phi_M^K$ to $Y$. 
\end{proof}

\subsection{Criterion} \label{sec:criterion}

We have seen in Theorem \ref{theo:properGK} a situation where the properness property of the moment maps $\Phi_M^G$ and 
$\Phi_M^K$ are equivalent. In this section, we start with a symplectic manifold $(M,\Omega_M)$  admitting an Hamiltonian action of a compact connected Lie group $K$. We suppose that the moment map $\Phi_M^{K}$ is {\em proper}. Let $K'\subset K$ be a closed subgroup. The aim of the section is to give a criterion under which the induced moment map $\Phi_M^{K'}$ is still proper. We start by recalling basic facts concerning the notion of asymptotic cone.

\medskip

To any non-empty subset $C$ of a real vector space $E$, we define its asymptotic cone 
$\As(C)\subset E$ as the set formed by the limits $y=\lim_{k\to\infty} t_k y_k$ where $(t_k)$ is a sequence of 
non-negative reals converging to $0$ and $y_k\in C$. Note that $\As(C)=\{0\}$ if and only if $C$ is compact.

We recall the following basic facts.
\begin{lem}\label{lem:As} Let $C_i, i=0,1$ be {\em closed} and {\em convex} subsets of $E$.

$\bullet$  We have $C_i + \As(C_i)\subset C_i$.

$\bullet$  If $C_0\cap C_1$ is non-empty we have $ \As(C_0)\cap  \As(C_1)= \As(C_0\cap C_1)$.

$\bullet$   If $C_0\cap C_1$ is non-empty and compact, we have $\As(C_0)\cap  \As(C_1)=\{0\}$
\end{lem}
\begin{proof}
Let us check the first point. Take $z\in C_i$ and $y=\lim_{k\to\infty} t_k y_k$ an element of $\As(C_i)$. 
Then $z+y=\lim_{k\to\infty}(1-t_k)z +t_k y_k$. Since $(1-t_k)z +t_k y_k\in C_i$ if $t_k\leq 1$, we know that $z+y\in C_i$ since $C_i$ is closed.

The inclusion $\As(C_0\cap C_1)\subset \As(C_0)\cap  \As(C_1)$ follows from the inclusions  $C_0\cap C_1\subset C_i$. Let 
$z\in C_0\cap C_1$ and $y\in \As(C_0)\cap  \As(C_1)$. Thanks to the first point we know that $z+\Rbb^{\geq 0}y\subset C_0\cap C_1$. Then 
$y=\lim_{t \to 0^+} t(z+t^{-1}y)\in \As(C_0\cap C_1)$. The second point is proved and the last point is a direct consequence of the second one.
\end{proof}


\medskip

The following Proposition is a useful tool for finding proper moment map. For a closed subgroup $K'$ of $K$, we denote 
$\pi_{\kgot',\kgot}:\kgot^*\to(\kgot')^*$ the projection which is the dual of the inclusion $\kgot'\croc\kgot$. The kernel 
$\pi_{\kgot',\kgot}^{-1}(0)$ is denoted $(\kgot')^\perp$.

\begin{prop}\label{prop:map=propre}
$\bullet$ Let $(M,\Omega_M)$ be an Hamiltonian $K$-manifold with a {\em proper} moment map $\Phi_M^{K}$. 
Let $\Delta_{K}(M)$ be its Kirwan polyhedron. Let $K'\subset K$ be a closed subgroup. Then the following statement are equivalent
\begin{itemize}
\item[a)] the moment map  $\Phi_M^{K'}=\pi_{\kgot',\kgot}\circ \Phi_M^{K}$ is {\em proper}, 
\item[b)]  $\As\left(\Delta_{K}(M)\right)\cap K\cdot (\kgot')^\perp=\{0\}$,
\item[c)] there exists $\varepsilon>0$, such that the inequality $\|\Phi_M^{K'}\|\geq\varepsilon \|\Phi_M^{K}\|-\varepsilon^{-1}$ holds on $M$.
\end{itemize}
\end{prop}

\begin{proof} If $c)$ does not hold we have a sequence $m_i\in M$ such that 
$\|\Phi_M^{K'}(m_i)\|\leq\frac{1}{i} \|\Phi_M^{K}(m_i)\|-i$, for all $i\geq 1$. Then $\|\Phi_M^{K}(m_i)\|$ tends to infinity and 
$\frac{\|\Phi_M^{K'}(m_i)\|}{\|\Phi_M^{K}(m_i)\|}$ tends to zero. We write $\Phi_M^{K}(m_i)=k_i\cdot y_i$ with $k_i\in K$ and 
$y_i\in \Delta_{K}(M)$. The sequence $\pi_{\kgot',\kgot}(k_i\cdot\frac{y_i}{\|y_i\|})$ converges to $0$. Here we can assume that the 
sequence $k_i$ converge to $k\in K$, and that the sequence 
$\frac{y_i}{\|y_i\|}$ converge to $y\in \As(\Delta_{K}(M))$. We get then that
$\pi_{\kgot',\kgot}(k\cdot y)=0$. In other words, $y$ is a non-zero element in 
$\As\left(\Delta_{K}(M)\right)\cap K\cdot(\kgot')^\perp$. We have proved $b)\Longrightarrow c)$.

The implication $c)\Longrightarrow a)$ is obvious. Let us prove the last implication $a)\Longrightarrow b)$. First we note that the properness of $\Phi_M^K$ implies that the projection $\pi_{\kgot',\kgot}$ is proper when restricted to the closed subset 
$\mathrm{Image}(\Phi_M^{K})=K\cdot \Delta_{K}(M)$. Let $k\in K$ and $\xi_o\in k\cdot \Delta_{K}(M)$. Then 
$$
k\cdot \Delta_{K}(M)\bigcap \xi_o +(\kgot')^\perp\subset \mathrm{Image}(\Phi_M^{K}) \bigcap \pi_{\kgot',\kgot}^{-1}(\pi_{\kgot',\kgot}(\xi_o))
$$
is non-empty and compact. If we apply the last point of Lemma \ref{lem:As} to the closed and convex sets $k\cdot \Delta_{K}(M)$ and  
$\xi_o +(\kgot')^\perp$ we get that 
$$
\As(k\cdot \Delta_{K}(M))\bigcap \As( \xi_o +\kgot^\perp)=k\cdot \As(\Delta_{K}(M))\bigcap (\kgot')^\perp
$$
is reduced to $\{0\}$. So we have proved that $ \As(\Delta_{K}(M))\bigcap k\cdot(\kgot')^\perp=\{0\}$ for any $k\in K$.
\end{proof}

\begin{rem}
When $M$ is a symplectic vector space $(E,\Omega_E)$, the moment map $\Phi_E^K:E\to \kgot^*$ is quadratic. Then
$\Phi_E^K$ is proper if and only if  $(\Phi_E^K)^{-1}(0)=\{0\}$.
\end{rem}

\subsection{Kostant-Souriau line bundle}\label{sec:line-bundle}

In the Kostant-Souriau framework, an Hamiltonian $G$-manifold
$(M,\Omega_M,\Phi^G_M)$ is pre-quantized if there is an equivariant
Hermitian line bundle $L_M$ with an invariant Hermitian connection
$\nabla_M$ such that
\begin{equation}\label{eq:kostant-L}
    \Lcal(X)-\iota(X_M)\nabla_M=i\langle\Phi^G_M,X\rangle\quad \mathrm{and} \quad
    (\nabla_M)^2= -i\Omega_M,
\end{equation}
for every $X\in\ggot$. 

The data $(L_M,\nabla_M)$ is also called a Kostant-Souriau line bundle. Note that 
conditions (\ref{eq:kostant-L}) imply via the equivariant
Bianchi formula the relations (\ref{eq:kostant=rel}).

\medskip

We suppose now that conditions {\bf C1} and {\bf C2} hold. Then $M=G\times_K Y$ where $Y\subset M$ is the 
$K$-invariant symplectic sub-manifold defined in Section \ref{sec:proper-moment-map}. Let $(L_M,\nabla_M)$ be a Kostant-Souriau 
line bundle on $M$. We denote $L_Y$ the restriction of the line bundle $L_M$ on $Y$. The connection $\nabla_M$ induces a $K$-invariant 
connection $\nabla_Y$ on $L_Y\to Y$, and we check easily that $(L_Y,\nabla_Y)$ is a Kostant-Souriau 
line bundle on $Y$.

Reciprocally, if $(L_Y,\nabla_Y)$ is a  Kostant-Souriau line bundle on $(Y,\Omega_Y,\Phi^K_Y)$, we define on $M$ the line bundle
$L_M:=(G\times L_Y)/ K$ equipped with the connection
$$
\nabla_M:=\nabla_Y + d^G+ i\langle\Phi^K_Y,\theta^K\rangle,
$$
where $d^G$ is the de Rham differential on $G$. Since $\Omega_M=\Omega_Y - d\langle\Phi^K_Y,\theta^K\rangle$, we check easily that $(L_M,\nabla_M)$ is a $G$-equivariant Kostant-Souriau line bundle on $(M,\Omega_M,\Phi^G_M)$.

\subsection{The case of elliptic orbits}

In this section, we consider the examples given by the {\em elliptic}  coadjoint orbits of $G$, that is  $M:= G\cdot \lambda$ for some $\lambda\in\kgot^*$. The Kirillov-Kostant-Souriau symplectic structure $\Omega_{M}$ is defined by the relation
$$
\Omega_{M}\vert_m(X_M\vert_m,Y_M\vert_m)=\langle m,[X,Y]\rangle,
$$
for $m\in M$ and $X,Y\in\ggot$. The corresponding moment map relatively to the action of $G$ on $G\cdot \lambda$ is the inclusion 
$\Phi^{G}_{M}:G\cdot \lambda\croc \ggot^*$. 

\begin{lem}\label{lem:orbite=moment}
The moment maps $\Phi^{G}_{M}$ and $\Phi^{K}_{M}$ are {\em proper}.
\end{lem}
\begin{proof}
The inclusion $\Phi^{G}_M$ is proper since the elliptic orbit $M= G\cdot \lambda$ is closed in $\ggot^*$. If we use the relations 
$\|\lambda\|^2=-b(\Phi^{G}_M, \Phi^{G}_M)= \|\Phi^{K}_M\|^2- \|\Phi^{\pgot}_M\|^2$ and $\|\Phi^{G}_M\|^2=\|\Phi^{K}_M\|^2+\|\Phi^{\pgot}_M\|^2$, we get $ \|\Phi^{K}_M\|^2=\frac{1}{2}(\|\lambda\|^2+ \|\Phi^{G}_M\|^2)$. The properness of $\|\Phi^{K}_M\|^2$ then follows.
 \end{proof}
 
 \medskip
 
We work now with an elliptic coadjoint orbit of $G$, $G\cdot \lambda$, such that the stabilizer subgroup $G_\lambda$ is {\em compact}.
Then the action of $G$ (and those of any closed subgroup) on $G\cdot\lambda$ is proper. 

Let $\tgot$ be the Lie algebra of a maximal torus $T$ in $K$. Our hypothesis concerning the compactness of $G_\lambda$ imposes $\tgot$ to be a Cartan sub-algebra of $\ggot$.  Let $\Rgot$ be the set of roots for the action of $\tgot$ on $\ggot\otimes\Cbb$. We have $\Rgot=\Rgot_c\cup \Rgot_n$  where $\Rgot_c$ and $\Rgot_n$ are respectively the set of roots for the action of $\tgot$ on 
$\kgot\otimes\Cbb$ and $\pgot\otimes\Cbb$. For the remaining part of this section, we fix a system of positive roots $\Rgot_{c,+}$ in 
$\Rgot_c$: let $\tgot^*_+\subset \tgot^*$ be the corresponding Weyl chamber. 

So $\lambda$ is chosen in the Weyl chamber $\tgot^*_+$, away from the non-compact wall~: $(\alpha,\lambda)\neq 0$ for all $\alpha\in \Rgot_n$. Thanks to Lemma \ref{lem:orbite=moment}, we know that the moment map $\Phi^{K}_{G\cdot \lambda}$ relative to a maximal compact subgroup $K\subset G$ is proper.  The Convexity theorem tells us that the set 
$$
\Delta_{K}(G\cdot\lambda):= \mathrm{Image}(\Phi^{K}_M)\cap \tgot^*_+
$$
is a closed convex locally polyhedral subset of $\tgot^*$. The results of Duflo-Heckman-Vergne \cite{DHV} shows that in fact 
$\Delta_{K}(G\cdot\lambda)$ is defined by a finite number of inequalities. In this paper, we call $\Delta_{K}(G\cdot\lambda)$ the Kirwan polyhedron.

 \medskip
 
We consider now a connected reductive subgroup $G'\subset G$, such that a Cartan involution $\Theta$ for $G$ leaves $G'$ invariant. Then we have Cartan decompositions $\ggot'=\kgot'\oplus\pgot'$ and $\ggot=\kgot\oplus\pgot$, with $\kgot'\subset \kgot$ and $\pgot'\subset\pgot$. Let $K\subset G$ and $K'\subset G'$ be the corresponding maximal subgroups.

Consider now the action of $G'$ on $(G\cdot\lambda,\Omega_{G\cdot\lambda})$. The moment map $\Phi^{G'}_{G\cdot\lambda}$ is the composition of the inclusion $G\cdot\lambda\croc \ggot^*$ with the orthogonal projection $\pi_{\ggot',\ggot}:\ggot^*\to(\ggot')^*$. Note that  $Y:=(\Phi^G_M)^{-1}(\kgot^*)$ is non-empty since it contains $\pi_{\ggot',\ggot}(\lambda)$.

We are looking to connected reductive subgroups $G'\subset G$ such that the moment map $\Phi^{G'}_{G\cdot \lambda}$ is proper. Theorem \ref{theo:properGK} shows that is equivalent to look at compact subgroups $K'\subset K$ such that the moment map 
$\Phi^{K'}_{G\cdot \lambda}$ is proper. Thanks to Proposition \ref{prop:map=propre}, we have the following criterium

\begin{prop}
$\bullet$ The moment map $\Phi^{G'}_{G\cdot \lambda}$ is proper if and only if \break $\As(\Delta_{K}(G\cdot\lambda))\cap K\cdot(\kgot')^\perp=\{0\}$.
\end{prop}

\medskip

We want to stress a property which is peculiar to the reductive Lie groups (in comparison with the nilpotent one).

\begin{prop}
Let $\Ocal\subset\ggot^*$ be a strongly elliptic orbit. Let $G'\subset G$ be a connected reductive subgroup such that 
$\Phi^{G'}_{\Ocal}=\pi_{\ggot',\ggot}:\Ocal\to(\ggot')^*$ is proper. Then for any coadjoint orbit $\Ocal'\subset(\ggot')^*$ 
the reduced space 
\begin{equation}\label{eq:fibre-connexe}
\Ocal\cap\pi_{\ggot',\ggot}^{-1}(\Ocal')/G'
\end{equation}
is connected.
\end{prop}

\begin{proof} Since $\Phi^{G'}_{\Ocal}$ is proper, we have a decomposition $\Ocal=G'\times_{K'} Y'$, where 
$Y'$ is a connected sub-manifold. Then $\Ocal\cap\pi_{\ggot',\ggot}^{-1}(\Ocal')$ is empty if $\Ocal'$ is not elliptic. And 
if $\Ocal'=G'\cdot\mu$ with $\mu\in(\kgot')^*$, we see that $\Ocal\cap\pi_{\ggot',\ggot}^{-1}(\Ocal')/G'\simeq (\Phi^K_Y)^{-1}(K\cdot\mu)/K$ which is connected since 
$\Phi^K_Y$ is proper (see the Convexity Theorem \cite{Atiyah82,Guillemin-Sternberg82.bis,Kirwan.84.bis,L-M-T-W}).
\end{proof}

\medskip

In general the Kirwan polyhedron $\Delta_{K}(G\cdot\lambda)$ is not known, but we can use at least the following observation. 
Let $\Rgot_n(\lambda)$ be the set of non-compact roots $\alpha$ such that $(\alpha,\lambda)>0$. Let us consider the following cone in 
$\tgot^*$ :
$$
\Ccal(\lambda):=\sum_{\alpha\in\Rgot_n(\lambda)} \Rbb^{\geq 0}\alpha.
$$

\begin{lem} \label{lem:Kirwan-G-lambda} For the Kirwan polyhedron we have
$$
\Delta_{K}(G\cdot\lambda)\subset \left\{\lambda + \Ccal(\lambda)\right\},\quad \mathrm{and\ then}\quad 
\As\left(\Delta_{K}(G\cdot\lambda)\right)\subset \Ccal(\lambda).
$$
\end{lem}

\begin{coro}\label{coro:moment-propre}
 The moment map $\Phi^{G'}_{G\cdot\lambda}$ is proper if  $\Ccal(\lambda)\, \cap\, K\cdot(\kgot')^\perp=\{0\}$.
\end{coro}

\begin{proof}  Let $C_{\lambda}$ be the cone tangent to $\Delta_{K}(G\cdot\lambda)$ at $\lambda$ : 
$$
C_{\lambda}=\Rbb^{\geq 0}\cdot\left\{\xi-\lambda,\ \xi\in \Delta_{K}(G\cdot\lambda)\right\}\subset \tgot^*.
$$
We have to show that $C_{\lambda}$ is contained in  $\Ccal(\lambda)$. Thanks to the result of Sjamaar \cite{Sjamaar-98}, we know that $C_{\lambda}$ is determined by a local Hamiltonian model near $K\cdot\lambda\subset G\cdot\lambda$.

The maximal torus $T$ of $K$ is still a maximal torus for the stabiliser subgroup $K_{\lambda}$ : let $\tgot^*_{\lambda,+}$ be a Weyl chamber for $(K_{\lambda}, T)$ which contains $\tgot^*_{+}$. Here, we consider the vector space $\pgot$ equipped with the linear symplectic structure $\Omega_{\lambda}(X,Y):=\langle\lambda,[X,Y]\rangle$. The group $K_{\lambda}$ acts in a Hamiltonian fashion on $(\pgot, \Omega_{\lambda})$. Let us denote by 
$$
\Delta_{K_{\lambda}}(\pgot)\subset \tgot^*_{\lambda,+}
$$
the corresponding Kirwan polytope (which is a rational cone). Since the stabiliser of the point $\lambda\in M:=G\cdot\lambda$ coincides with the stabiliser subgroup $K_{\lambda}$ of its image by the moment map $\Phi_M^{K}$, the local form of Marle \cite{Marle85} and Guillemin and Sternberg \cite{Guillemin-Sternberg84} tells us that $M$ is symplectomorphic with $K\times_{K_{\lambda}}\pgot$ in a neighbourhood of $K\cdot\lambda$. Theorem 6.5 of  \cite{Sjamaar-98} tells us then that $C_{\lambda}= \Delta_{K_{\lambda}}(\pgot)$. 

Let us consider the Hamiltonian action of the torus $T$ on $(\pgot, \Omega_{\lambda})$. Let $J_\lambda$ be an invariant complex structure on  $\pgot$ which is compatible with $\Omega_{\lambda}$: we can check that the weights of the $T$-action on $(\pgot,J_\lambda)$ are 
$-\alpha$, for $\alpha\in \Rgot_n(\lambda)$. Hence the image $\Delta_{T}(\pgot)$ of the moment map is equal to the cone generated by the weights $\alpha\in \Rgot_n(\lambda)$. Finally  we have proved that
$$
C_{\lambda}=\Delta_{K_{\lambda}}(\pgot)\subset \Delta_{T}(\pgot)=\Ccal(\lambda).
$$
\end{proof}

\section{Quantization commutes with reduction}\label{sec:QR-hol}

Let $G$ be a connected real reductive Lie group and let $K$ be a maximal connected compact subgroup. Let $\cgot_\kgot, \cgot_\ggot$ be respectively the center of $\kgot$ and $\ggot$. In all the section we assume that the group $G$ satisfies the 
following condition 
\begin{equation}\label{condition-hol}
Z_\ggot(\cgot_\kgot)=\kgot,
\end{equation}
i.e. the centralizer of $\cgot_\kgot$ in $\ggot$ coincides with $\kgot$. Hence $\cgot_\ggot\subset \cgot_\kgot\subset \kgot$.

We make the choice of a maximal torus $T$ in $K$ with Lie algebra $\tgot$. Note that (\ref{condition-hol}) forces $\tgot$ to be a Cartan sub-algebra of $\ggot$. Let $\Rgot=\Rgot_c\cup \Rgot_n$ be the set of roots. We fix a system of positive roots $\Rgot_c^+$ in 
$\Rgot_c$. We know also that (\ref{condition-hol}) imposes the existence of elements $z\in \cgot_\kgot$ such that $\mathrm{ad}(z)$ defines a complex structure on $\pgot$ (see Section 9 in \cite{Knapp-book}).  For such element $z$, we define 
$$
\Rgot_n(z):=\{\alpha \in \Rgot_n\ \vert \ \langle\alpha,z\rangle=1\}.
$$
which is invariant relatively to the action of the Weyl group $W_K$. The union $\Rgot_c^+\cup \Rgot_n(z)$ defines then a system of positive roots in $\Rgot$. 

We will be interested to the closed $W_K$-invariant cones in $\tgot^*$
\begin{eqnarray*}
\Chol(z)&:=&\left\{ \xi\in\tgot^*\  \vert\ (\beta,\xi)\geq 0, \  \forall \beta\in \Rgot_n(z)\right\},\\
\Ccal(z)&:=&\sum_{\beta\in \Rgot_n(z)} \Rbb^{\geq 0}\beta.
\end{eqnarray*}

We recall some basic facts about them. 
\begin{lem}We have the following inclusions
\begin{equation}\label{eq:inclusion-cone}
\Ccal(z)\subset\Chol(z)\subset \{\xi\in\tgot^*\  \vert\ \langle\xi,z\rangle\geq 0 \}.
\end{equation}
\end{lem}
\begin{proof}
Since $(\beta_0,\beta_1)\geq 0$ for any $\beta_k\in \Rgot_n(z)$, we see that $\Ccal(z)\subset\Chol(z)$. 
For $\xi\in\tgot^*$, we have $\langle\xi, z\rangle=-b(\tilde{\xi}, z)= 2\sum_{\beta\in\Rgot_n(z)}\langle\beta, \tilde{\xi}\rangle$ with $\langle\beta, \tilde{\xi}\rangle=(\beta,\xi)$. Then $\xi\in\Chol(z)$ implies $\langle\xi, z\rangle\geq 0$.
\end{proof}

\medskip

\begin{exam}\label{example:holomorphic}
We have the following classical examples:
\begin{center}
\begin{tabular}{ccc}
$G $                                   &                             $K$             &   $\pgot$    \\
 \hline
   & \\
$\mathrm{Sp}(n,\Rbb)$    & $\mathrm{U}(n)$                      				&  $S^2(\Cbb^n)$         \\
$\mathrm{SO}^*(2n)$       & $\mathrm{U}(n)$   						&   $\wedge^2\Cbb^n$  \\
 $\mathrm{SO}_o(2,n)$    & $\mathrm{SO}(2)\times\mathrm{SO}(n)$ 	&    $\Cbb^n$                    \\
$\mathrm{U}(p,q)$            & $\mathrm{U}(p)\times\mathrm{U}(p)$               &  $M_{p,q}(\Cbb)$ \\
\end{tabular}
\end{center}

\end{exam}


\subsection{Holomorphic coadjoint orbits}

The {\em holomorphic} coadjoint orbits are $G\cdot\lambda$ with $\lambda$ in the interior of $\Chol(z)$. These symplectic manifolds possess
a $G$-invariant (integrable) complex structure $J_\lambda$ which is compatible with the symplectic structure $\Omega_{G\cdot\lambda}$ (see \cite{pep-hol-series}). Hence $(G\cdot\lambda, \Omega_{G\cdot\lambda}, J_\lambda)$ is a K\"ahler manifold when $\lambda\in\mathrm{Interior}(\Chol(z))$.

The real $K$-module $\pgot$ is equipped with the invariant linear symplectic structure $\Omega_\pgot(A,B):=-b(z,[A,B])$. We have two families of Hamiltonian $K$-manifold~: $K\cdot\lambda\times\pgot$ and $G\cdot\lambda$ for $\lambda\in\Chol(z)$. We start with the fundamental fact.

\begin{prop}\label{prop:Delta-K-Lambda}
Let $\lambda\in\mathrm{Interior}(\Chol(z))$. We have 
\begin{enumerate}
\item[a)]  $\Delta_{K}(G\cdot\lambda)\subset \lambda+\Ccal(z)\subset \Chol(z)$, 
\item[b)] $\Delta_{K}(G\cdot\lambda)=\Delta_{K}(K\cdot\lambda\times \pgot)$,
\item[c)]  $\As\left(\Delta_{K}(G\cdot\lambda)\right)=\Delta_{K}(\pgot)$.
\end{enumerate}
\end{prop}

\begin{proof} Point $a)$ is the translation of Lemma \ref{lem:Kirwan-G-lambda} since the cone $\Ccal(\lambda)$ is equal to $\Ccal(z)$. 
Point $b)$ is proved in \cite{pep-hol-series}. Another proof is given by Deltour in \cite{Deltour-mcduff}, by showing the stronger result that the Hamiltonian $K$-manifolds  $G\cdot\lambda$ and $K\cdot\lambda\times \pgot$ are symplectomorphic. The point $c)$ follows easily from $b)$.
\end{proof}

\begin{rem}
When $G$ is one of the groups appearing in Example \ref{example:holomorphic}, the generators of the cone $\Delta_{K}(\pgot)$ can be defined in term of \emph{strongly orthogonal roots} (see Section 5 in \cite{pep-hol-series}). Note also that Deltour has completely described the facet of the polytopes $\Delta_{K}(G\cdot\lambda)$ when $\mathrm{Interior}(\Chol(z))$ (see \cite{Deltour-transf-group}).
\end{rem}

\medskip

Let $S^\bullet(\pgot)$ be the symmetric algebra of the complex $K$-module $(\pgot,\mathrm{ad}(z))$: it is an admissible representation of 
$K$. Let $K'$ be a closed connected subgroup of $K$. We denote by $\Phi^{K'}_{G\cdot\Lambda}$ and by  $\Phi^{K'}_{\pgot}$ the corresponding moment maps. We have the following

\begin{prop}\label{prop:Phi-K-hol-propre}
Let $\lambda\in\mathrm{Interior}(\Chol(z))$. The following assertions are equivalent
\begin{enumerate}
\item[a)]  $\Phi^{K'}_{G\cdot\lambda}:G\cdot\lambda\to (\kgot')^*$ is a proper map,
\item[b)]  $\Delta_{K}(\pgot)\cap K\cdot (\kgot')^\perp=\{0\}$,
\item[c)]  $\Phi^{K'}_{\pgot}:\pgot\to (\kgot')^*$ is a proper map,
\item[d)]  $\{\Phi^{K'}_{\pgot}=0\}$ is reduced to $\{0\}$,
\item[e)]  $S^\bullet(\pgot)$ is an admissible representation of $K'$.
\end{enumerate}
\end{prop}

\begin{proof}The equivalences $a)\Leftrightarrow b)$ and $b)\Leftrightarrow c)$ follow from Propositions \ref{prop:map=propre} and \ref{prop:Delta-K-Lambda}. The other equivalences $c)\Leftrightarrow d)\Leftrightarrow e)$ are proved in \cite{pep-formal}[Section 5].
\end{proof}

\bigskip

Let us consider the moment map $\Phi^{K}_{\pgot}:\pgot\to \kgot^*$. Via the identification $\kgot^*\simeq \kgot$, the moment map $\Phi^{K}_{\pgot}$ is defined by
$$
\Phi^{K}_{\pgot}(X)= -[X,[z,X]],\quad X\in\pgot.
$$
Hence we see that $\langle\Phi^{K}_{\pgot},z\rangle : \pgot\to\Rbb$ is a \emph{proper map} taking positive values. This simple fact and Proposition \ref{prop:Phi-K-hol-propre} gives us the following 

\begin{coro}\label{coro:propre-z}
Let $G'$ be a connected reductive subgroup of $G$, and let $\lambda\in\mathrm{Interior}(\Chol(z))$. The moment map $\Phi^{G'}_{G\cdot\Lambda}$ is proper when 
the Lie algebra $\ggot'$ contains $\Rbb z$.
\end{coro}

\begin{exam} The condition $\Rbb z\subset \ggot'$ is fulfilled in the following cases:
\begin{enumerate}
\item $G'=\mathrm{SO}_o(2,p)\subset G=\mathrm{SO}_o(2,n)$ for $0\leq p\leq n$,
\item $G'$ is the identity component of $G^\sigma$, where $\sigma$ is an involution of $G$ such that $\sigma(z)=z$ (see Table \ref{tab:involution}),
\item $G'$ is the diagonal in $G:=G'\times\cdots\times G'$.
\end{enumerate}
\end{exam}

\begin{table}
   \caption{\label{tab:involution} Involution $\sigma$ such that $\sigma(z)=z$}
\begin{center}
\begin{tabular}{c|c}
$G $                               &                             $G^\sigma$               \\
 \hline
   & \\
$\mathrm{Sp}(n,\Rbb)$ & $\mathrm{Sp}(p,\Rbb)\times \mathrm{Sp}(n-p,\Rbb)$ \\
 $\mathrm{Sp}(n,\Rbb)$ & $\mathrm{U}(p,n-p)$ \\
 $\mathrm{SO}(2,2n)$ & $\mathrm{U}(1,n)$ \\
$\mathrm{SO}(2,n)$ & $\mathrm{SO}(2,p)\times\mathrm{SO}(n-p)$ \\
$\mathrm{SO}^*(2n)$ & $\mathrm{U}(p,n-p)$ \\
$\mathrm{SO}^*(2n)$ & $\mathrm{SO}^*(2p)\times\mathrm{SO}^*(2n-2p)$ \\
$\mathrm{U}(n,n)$ & $\mathrm{Sp}(n,\Rbb)$ \\
$\mathrm{U}(n,n)$ & $\mathrm{SO}^*(2n)$ \\
$\mathrm{U}(p,q)$ & $\mathrm{U}(i,j)\times\mathrm{U}(p-i,q-j)$ \\
\end{tabular}
\end{center}
\end{table}

\subsection{Holomorphic discrete series}

Let $\wedge^*\subset \tgot^*$ be the lattice of characters of $T$. We know that the set $\wedge^*_+:=\wedge^*\cap\tgot^*_+$ parametrizes the set $\widehat{K}$ of irreducible representations of $K$: for any $\mu\in\wedge^*_+$, we denote $V_\mu^K$ the irreducible representation of $K$ with highest weight $\mu$.

We will be interested in $\chol(z)=2\rho_n(z) + \Chol(z)\subset\Chol(z)$ where $2\rho_n(z)$ is the sum of the roots of $\Rgot_n(z)$. Let us denote
$$
\Ghol(z):=\wedge^*_+\bigcap \Chol^\rho(z).
$$

\begin{theo}[Harish-Chandra]
For any $\lambda\in\Ghol(z)$, there exists a irreducible unitary representation of $G$, denoted $V^G_\lambda$, such that 
the vector space of $K$-finite vectors is 
$$
V^G_\lambda\vert_K:=V_\lambda^K\otimes S^\bullet(\pgot).
$$
Here $S^\bullet(\pgot)$ is the symmetric algebra of the complex vector space $(\pgot,\mathrm{ad}(z))$.
\end{theo}

\subsection{Formal geometric quantization}\label{subsec:formal-quantization}

Let us first recall the definition of the geometric quantization of a
smooth and compact Hamiltonian manifold. Then we show a way of extending 
the notion of geometric quantization to the case of a \emph{non-compact} 
Hamiltonian manifold. 

\medskip

Let $K$ be a compact connected Lie group. Let  $(M,\Omega_M,\Phi^K_M)$ be a Hamiltonian $K$-manifold which is pre-quantized by the  
Hermitian line bundle $L_M$ (see Section \ref{sec:line-bundle}).

Let us recall the notion of geometric quantization when $M$ is \textbf{compact}.
Choose a $K$-invariant almost complex structure $J$ on $M$ which is
compatible with $\Omega_M$ in the sense that the symmetric bilinear
form $\Omega_M(\cdot,J\cdot)$ is a Riemannian metric. Let
$\overline{\partial}_{L_M}$ be the Dolbeault operator with coefficients
in $L$, and let $\overline{\partial}_{L_M}^*$ be its (formal) adjoint.
The \emph{Dolbeault-Dirac operator} on $M$ with coefficients in ${L_M}$
is $D_{L_M}=\sqrt{2}( \overline{\partial}_{L_M}+\overline{\partial}_{L_M}^*)$, considered
as an elliptic operator from $\Acal^{0,\textrm{\tiny even}}(M,L_M)$ to
$\Acal^{0,\textrm{\tiny odd}}(M,L_M)$. Let $R(K)$ be the representation ring of $K$.

\begin{defi}\label{def:quant-compact-lisse}
 The geometric quantization of a {\em compact} Hamiltonian $K$-manifold $(M,\Omega_M,\Phi^K_M)$ is the element 
 $\Qcal_K(M)\in R(K)$ defined as the equivariant index of the
Dolbeault-Dirac operator $D_{L_M}$.
\end{defi}

\medskip 

Let us consider the case of a \textbf{proper} pre-quantized Hamiltonian $K$-manifold $M$:  the manifold is (perhaps) \textbf{non-compact} but 
the moment map $\Phi^K_M: M\to \kgot^{*}$ is supposed to be proper. In this setting, we have two ways of extending the 
geometric quantization procedure.

$\mathbf{First \ way :\  \qfor_K.} $\ One defines the {\em formal geometric quantization} of $M$ as an element $\qfor_K(M)$ that belongs to $\Rfor(K):=\hom_\Zbb(R(K),\Zbb)$ \cite{Weitsman,pep-formal,Ma-Zhang,pep-formal-2}. Let us recall the definition. 
\medskip

%

For any $\mu\in \widehat{K}$ which is a regular value of the moment map $\Phi$, the
reduced space\footnote{The symplectic quotient will be denoted $M_{\mu,K}$ when we need more precise notations.} (or symplectic quotient)
\begin{equation}\label{eq:symplectic-quotient}
M_{\mu}:=(\Phi^K_M)^{-1}(K\cdot\mu)/K
\end{equation}
is a {\em compact} orbifold equipped with a symplectic structure $\Omega_{\mu}$. Moreover
$L_{\mu}:=(L\vert_{(\Phi^K_M)^{-1}(\mu)}\otimes \Cbb_{-\mu})/K_{\mu}$ is a
Kostant-Souriau line orbibundle over $(M_{\mu},\Omega_{\mu})$. The
definition of the index of the Dolbeault-Dirac operator carries over
to the orbifold case, hence $\Qcal(M_{\mu})\in \Zbb$ is defined. This notion of geometric 
quantization extends further to the case of singular symplectic quotients \cite{Meinrenken-Sjamaar,pep-RR}. So the integer 
$\Qcal(M_{\mu})\in \Zbb$ is well defined for every $\mu\in \widehat{K}$: in particular
$\Qcal(M_{\mu})=0$ if $\mu$ is not in the Kirwan polytope $\Delta_K(M)$.

\begin{defi}\label{def:formal-quant}
Let $(M,\Omega_M,\Phi^K_M)$ be a {\em proper} Hamiltonian $K$-manifold which is 
pre-quantized by a Kostant-Souriau line bundle $L$. The formal quantization of $(M,\Omega_M,\Phi^K_M)$ is 
the element of $\Rfor(K)$ defined by
$$
\qfor_K(M)=\sum_{\mu\in \widehat{K}}\Qcal(M_{\mu})\, V_{\mu}^{K}\ .
$$
\end{defi}

\medskip

When $M$ is compact, the fact that 
\begin{equation}\label{eq:Q-R=0}
\Qcal_K(M)=\qfor_K(M)
\end{equation}
is known as the ``quantization commutes with reduction'' Theorem. This  was conjectured 
by Guillemin-Sternberg in \cite{Guillemin-Sternberg82} and was first proved by Meinrenken 
\cite{Meinrenken98} and  Meinrenken-Sjamaar \cite{Meinrenken-Sjamaar}. Other proofs 
of (\ref{eq:Q-R=0}) were also given by Tian-Zhang \cite{Tian-Zhang98} and the author 
\cite{pep-RR}. For complete references on the subject the reader should consult 
\cite{Sjamaar96,Vergne-Bourbaki}.

One of the main features of the formal geometric quantization $\qfor$ is summarized in the following 

\begin{theo}[\cite{pep-formal}]\label{theo:pep-formal}

$\bullet$ \emph{\bf Restriction to subgroup.} Let $M$ be a pre-quantized Hamiltonian $K$-manifold which is \emph{proper}. 
Let $H\subset K$ be a closed connected Lie subgroup such that $M$ is still \emph{proper} as a
Hamiltonian $H$-manifold. Then $\qfor_K(M)$ is $H$-admissible and we
have $\qfor_K(M)|_H=\qfor_H(M)$ in $\Rfor(H)$.

\medskip

$\bullet$ \emph{\bf Product.} Let $M$ and $N$ be pre-quantized Hamiltonian $K$-manifolds with $M$ is \emph{proper} and $N$ is \emph{compact}. Then $M\times N$ is a \emph{proper} pre-quantized Hamiltonian $K$-manifold and we have $\qfor_K(M\times N)=\qfor_K(M)\cdot \Qcal_K(N)$ in $\Rfor(K)$.
\end{theo}

\medskip 

$\mathbf{Second\ way :\ \Qcal^{\Phi}_K.}$\  When $M$ is a proper pre-quantized Hamiltonian $K$-manifold, we can define another 
{\em formal geometric quantization} of $M$ through a non-abelian localization procedure \`a la Witten \cite{Witten}. In  \cite{Ma-Zhang,pep-formal}, one proves that an element 
\begin{equation}\label{eq:Q-Phi}
\Qcal^{\Phi}_K(M)\in \Rfor(K)
\end{equation}
is well-defined by localizing the index of the Dolbeault-Dirac operator $D_{L_M}$ on the set $\Cr(\|\Phi^K_M\|^2)$ of 
critical points of the square of the moment map.  

The crucial result is that these two procedures coincides \cite{Ma-Zhang,pep-formal}.

\begin{theo}[Ma-Zhang, Paradan]\label{theo:formel=formel}
Let $M$ be a proper pre-quantized Hamiltonian $K$-manifold. Then, the following equality 
\begin{equation}\label{eq:qfor=qfor}
\qfor_K(M)=\Qcal^{\Phi}_K(M).
\end{equation}
holds in $\Rfor(K)$.
\end{theo}

\subsection{Formal geometric quantization of holomorphic orbits}

Let us come back to the holomorphic discrete representation $V^G_\lambda$. Consider a coadjoint orbit $G\cdot\lambda$ for $\lambda\in \wedge^*_+$ in the interior of the chamber $\Chol(z)$, so that $\lambda$ is strongly elliptic. The action of $G$ on $G\cdot\lambda$ is Hamiltonian, and the line bundle 
$$
L:=G\times_{K_{\lambda}}\Cbb_{\lambda}
$$
is a Kostant-Souriau line bundle over $G\cdot\lambda\simeq G/K_{\lambda}$. Here $\Cbb_{\lambda}$ denotes the $1$-dimensional representation of the stabilizer subgroup $K_\lambda$ that can be attached to the weight $\lambda$. 

Thanks to Lemma \ref{lem:orbite=moment}, we know that the moment map $\Phi^K_{G\cdot\lambda}$ relatively to the action of $K$ on  $G\cdot\lambda$ is proper. Hence the reduced spaces 
$$
(G\cdot\lambda)_\mu:=(\Phi^K_{G\cdot\lambda})^{-1}(K\cdot\mu)/K.
$$
are compact for any $\mu\in\wedge^*_+$, and the generalized character $\Qcal^\Phi_K(G\cdot\lambda)\in\Rfor(K)$ is well defined. We have proved in \cite{pep-ENS,pep-hol-series} the following 

\begin{theo}\label{theo:Q-Phi-G-lambda}
Let $\lambda\in \wedge^*_+\cap\mathrm{Interior}(\Chol(z))$. The following equality
$$
\Qcal^\Phi_K(G\cdot\lambda)=V^K_\lambda\otimes S^\bullet(\pgot)
$$
holds in $\Rfor(K)$.
\end{theo}

This result will be generalized in (\ref{eq:formal-G-K}). It shows that $\Qcal^\Phi_K(G\cdot\lambda)$ coincides with the vector space of $K$-finite vector of the holomorphic discrete representation $V^G_\lambda$ when $\lambda\in\chol(z)$. Note that for $\lambda\in \mathrm{Interior}(\Chol(z))\setminus\chol(z)$, the generalized character $\Qcal^\Phi_K(G\cdot\lambda)$ can not be associated to an holomorphic discrete representation of $G$. 

Theorems \ref{theo:Q-Phi-G-lambda} and \ref{theo:formel=formel}, gives us the following informations concerning the $K$-multiplicities.

\begin{coro}\label{cor:K-multiplicity}Let $\lambda\in \wedge^*_+\cap\mathrm{Interior}(\Chol(z))$, and $\mu\in\wedge^*_+$. 

$\bullet$ The multiplicity of $V^K_\mu$ in $V^K_\lambda\otimes S^\bullet(\pgot)$ is equal to the quantization of the reduced space $(G\cdot\lambda)_\mu$.

$\bullet$ If $[V^K_\mu : V^K_\lambda\otimes S^\bullet(\pgot)]\neq 0$ then $\mu\in \lambda + \Ccal(z)\subset\Chol(z)$. The last condition imposes that $\|\mu\|> \|\lambda\|$ or $\mu=\lambda$.
\end{coro}

\begin{proof}
The first point is a consequence of the equality $\Qcal^\Phi_K(G\cdot\lambda)=$ \break $\qfor_K(G\cdot\lambda)$. We know then that if 
$[V^K_\mu : V^K_\lambda\otimes S^\bullet(\pgot)]\neq 0$, then $\mu$ belongs to the Kirwan polytope $\Delta_K(G\cdot \lambda)$. 
But we know after Lemma \ref{lem:Kirwan-G-lambda} that $\Delta_K(G\cdot \lambda)\subset \lambda +\Ccal(z)$ : so 
$\mu=\lambda +\sum_{\beta\in\Rgot_n(z)} x_\beta \beta$ with $x_\beta\geq 0$. Finally, we have
\begin{eqnarray*}
\|\mu\|^2 &=& \|\lambda\|^2 + \|\!\!\sum_{\beta\in\Rgot_n(z)} x_\beta \beta\,\|^2 +2 \sum_{\beta\in\Rgot_n(z)} x_\beta 
\underbrace{(\beta,\lambda)}_{\geq 0}\\
&\geq& \|\lambda\|^2,
\end{eqnarray*}
and we have $\|\mu\|^2= \|\lambda\|^2$ only if $\mu=\lambda$.
  \end{proof}

\subsection{Multiplicities of the holomorphic discrete series}\label{subsection:multiplicity-hol-discrete}

We consider now a connected reductive subgroup $G'\subset G$ such that $z\in \ggot'$. Then it is easy to check that $G'$ satisfies 
(\ref{condition-hol}). Let $K'\subset K$ be the maximal compact subgroup in $G'$, and let $T'\subset T$ be a maximal torus in $K'$. 
Let $\Chol(z),\chol(z)\subset \tgot^*$ and $\Chol'(z),\cholp(z)\subset (\tgot')^*$ be the corresponding convex cone.
For $\lgot\in\{\tgot,\kgot,\ggot\}$, we denote $\pi_{\lgot',\lgot}:\lgot^*\to(\lgot')^*$ the canonical projection. We have the 
following important  fact.

\begin{prop}\label{prop:cone-projection}
We have the following relations
\begin{eqnarray*}
(a)\;\quad\qquad\pi_{\tgot',\tgot}\left(\Chol(z)\right)&\subset& \Chol'(z),\\
(b)\qquad\pi_{\kgot',\kgot}\left(K\cdot\Chol(z)\right)&\subset& K'\cdot\Chol'(z),\\
(c)\qquad\pi_{\kgot',\kgot}\left(K\cdot\chol(z)\right)&\subset& K'\cdot\cholp(z),\\
(d)\qquad\pi_{\ggot',\ggot}\left(G\cdot\chol(z)\right)&\subset& G'\cdot\cholp(z).
\end{eqnarray*}
\end{prop}

\begin{proof} Let $\alpha\in \tgot^*$ be a non-compact root of $(\ggot,\tgot)$. Let $\ggot_\alpha\subset \pgot\otimes\Cbb$ 
the corresponding $1$-dimensional weight space. Then we know that there exists $h_\alpha\in i[\ggot_\alpha,\overline{\ggot_\alpha}]\cap \tgot$
such that $\alpha= -b(h_\alpha,\cdot)$. Note that the half-line $\Rbb^{>0}h_\alpha$ does not depend of the bilinear form $b$, and that 
the condition $(\alpha,\xi)\geq 0$ is equivalent to $\langle\xi,h_\alpha\rangle\geq 0$ for any $\xi\in\tgot^*$.

Let $\alpha\in \tgot^*$ be a non-compact root of $(\ggot,\tgot)$ such that its restriction $\alpha'=\pi_{\tgot',\tgot}(\alpha)$ is a non-compact root of $(\ggot',\tgot')$. Since the $1$-dimensional weight spaces $\ggot_\alpha$ and $\ggot'_{\alpha'}$ coincide we have $\Rbb^{>0}h_\alpha=\Rbb^{>0}h_{\alpha'}\subset\tgot'$. Then the condition $\langle\xi,h_\alpha\rangle\geq 0$ is equivalent to 
$\langle\pi_{\tgot',\tgot}(\xi),h_{\alpha'}\rangle\geq 0$. Finally we have proved the point $(a)$ : if one has $\langle\xi,h_\alpha\rangle\geq 0$
for any positive non-compact root of $(\ggot,\tgot)$, then $\langle\pi_{\tgot',\tgot}(\xi),h_{\alpha'}\rangle\geq 0$ holds for any positive non-compact root of $(\ggot',\tgot')$. 

Let $\xi\in\Chol(z)$ and $\xi'\in\pi_{\kgot',\kgot}(K\cdot\xi)\cap (\tgot')^*$. Then 
$\xi'\in \pi_{\tgot',\tgot}\circ \pi_{\tgot,\kgot}(K\cdot\xi)$. By the Convexity theorem 
\cite{Atiyah82,Guillemin-Sternberg82.bis,Kirwan.84.bis,L-M-T-W}, we know that $\pi_{\tgot,\kgot}(K\cdot\xi)$
is equal to the convex hull of $W_K\xi$. But $\xi$ belongs to the $W_K$-invariant convex cone $\Chol(z)$, and then 
$\pi_{\tgot,\kgot}(K\cdot\xi)\subset \Chol(z)$. Finally $\xi'\in \pi_{\tgot',\tgot}(\Chol(z))\subset \Chol'(z)$ thanks to the point $(a)$.

Let $\xi\in\Chol(z)$. Since $2\rho_n(z)$ is $K$-invariant, we have $K\cdot(2\rho_n(z)+\xi)=2\rho_n(z)+K\cdot\xi$. 
Thanks to the point $(b)$, we see that
\begin{eqnarray*}
\pi_{\kgot',\kgot}(K\cdot(2\rho_n(z)+\xi))&=&\pi_{\kgot',\kgot}(2\rho_n(z))+ \pi_{\kgot',\kgot}(K\cdot\xi)\\
&\subset&K'\cdot\left(\pi_{\kgot',\kgot}(2\rho_n(z))+ \Chol'(z)\right).
\end{eqnarray*}
The $K'$-invariant term $\pi_{\kgot',\kgot}(2\rho_n(z))$ belongs to $(\tgot')^*$ and is equal to $2\rho_n'(z)+\pi_{\tgot',\tgot}(A)$ 
where $A$ is the sum of the positive non-compact roots $\alpha$ such that $\ggot_\alpha$ is not included in $\pgot'\otimes\Cbb$. Hence
$A\in \Chol(z)$ and thanks to point $(a)$ its projection $\pi_{\tgot',\tgot}(A)$ belongs to $\Chol'(z)$.
The point $(c)$ is then proved.

Let $\lambda\in\chol(z)$. The coadjoint orbit $G\cdot\lambda$ is contained in $\ggot^*_{se}$, and the moment map 
$\Phi_{G\cdot\lambda}^{G'}$ is proper since $z\in \ggot'$ (see Corollary \ref{coro:propre-z}). Then, we know that 
$$
\pi_{\ggot',\ggot}\left(G\cdot\lambda\right)=\mathrm{Image}(\Phi_{G\cdot\lambda}^{G'})=
G'\cdot\left(\pi_{\ggot',\ggot}\left(G\cdot\lambda\right)\bigcap (\kgot')^*\right)
$$
and 
\begin{eqnarray*}
\pi_{\ggot',\ggot}\left(G\cdot\lambda\right)\bigcap (\kgot')^* &\subset &\pi_{\kgot',\kgot}\circ\pi_{\kgot,\ggot}\left(G\cdot\lambda\right)\\
&\subset &\pi_{\kgot',\kgot}\left(K\cdot\Delta_K(G\cdot\lambda)\right)\\
&\subset &\pi_{\kgot',\kgot}\left(K\cdot\chol(z)\right) \qquad [1]\\
&\subset &  K'\cdot\cholp(z).\quad \qquad [2]
\end{eqnarray*}
Equality $[1]$ is due to the fact that $\Delta_K(G\cdot\lambda)\subset \lambda+\Ccal(z)\subset \chol(z)$ when $\lambda\in\chol(z)$ (see Lemma \ref{prop:Delta-K-Lambda}). Equality $[2]$ corresponds to $c)$.
\end{proof}

\medskip

\begin{rem}
When the Lie algebra $\ggot$ is simple the set $G\cdot\Chol(z)\subset \ggot_{se}^*$ is a maximal closed convex $G$-invariant cone. See 
\cite{Paneitz83,Vinberg80}.
\end{rem}

\medskip

We finish the section by considering the restriction of the irreducible representation $V^G_\lambda$ to the reductive subgroup $G'$.  
We will denoted $\Khol(z)\subset \widehat{K}$ the subset $\wedge^*_+\cap\chol(z)$. We see that $\Khol(z)$ and $\Ghol(z)$ are the same set 
but they parametrize representations of different groups ($K$ and $G$ respectively).

We start with the

\begin{prop}\label{prop:restriction-G-K-prime}
 Let $b\in \Khol(z)$, $\lambda\in \Ghol(z)$ and $\mu\in \widehat{K'}$. We have

$\bullet$ If  $[V_\mu^{K'} : V^K_b\vert_{K'}]\neq 0$ then $\mu\in \Kholp(z)$.

$\bullet$ If $[V_\mu^{K'} : V^G_\lambda\vert_{K'}]\neq 0$ then $\mu\in \Gholp(z)$.
\end{prop}

\begin{proof} We use here the "Restriction to subgroup" property of Theorem \ref{theo:pep-formal}.

For the first point, we know after the Borel-Weil Theorem that $V^K_b=\Qcal_K(K\cdot b)$, and then  
$V^K_b\vert{K'}=\Qcal_{K'}(K\cdot b)$. Then $[V_\mu^{K'} : V^K_b\vert_{K'}]\neq 0$ only if $b$ belongs to 
$\mathrm{Image}(\Phi^{K'}_{K\cdot b})=\pi_{\kgot',\kgot}(K\cdot b)\subset K'\cdot\cholp(z)$. But 
$K'\cdot\cholp(z)(\tgot')^*=\cholp(z)$ since $\cholp(z)$ is $W_{K'}$-invariant. We have proved finally that 
$b\in \Kholp(z)$.

For the second point, it works the same. We know that $V^G_\lambda\vert_{K'}=\Qcal^\Phi_{K'}(G\cdot\lambda)$. Hence $V_\mu^{K'}$ occurs in the restriction $V_\lambda^G\vert_{K'}$ only if $\mu$ belongs to the image of the moment map $\Phi^{K'}_{G\cdot\lambda}=\pi_{\kgot',\kgot}\circ\Phi^{K}_{G\cdot\lambda}$. 
Since $\Delta_K(G\cdot\lambda)\subset \chol(z)$ (see Lemma \ref{prop:Delta-K-Lambda}), we have 
\begin{eqnarray*}
\mathrm{Image}(\Phi^{K'}_{G\cdot\lambda})=\pi_{\kgot',\kgot}\Big(\mathrm{Image}(\Phi^{K}_{G\cdot\lambda})\Big)
&=& \pi_{\kgot',\kgot}\Big(K\cdot\Delta_K(G\cdot\lambda)\Big)\\
&\subset & \pi_{\kgot',\kgot}\Big(K\cdot\chol(z))\Big)\subset K'\cdot \cholp(z),
\end{eqnarray*}
where the last inclusion is point $(c)$ of Proposition \ref{prop:cone-projection}. We have then proved that $V_\mu^{K'}$ occurs in the restriction $V_\lambda^G\vert_{K'}$ only if
$$
\mu \in \mathrm{Image}(\Phi^{K'}_{G\cdot\lambda})\cap (\tgot')^*\subset (K'\cdot \cholp(z))\cap (\tgot')^*=\cholp(z).
$$
The last equality is due to the fact that $\cholp(z)$ is a $W_{K'}$-invariant subset of $(\tgot')^*$.
\end{proof}

\medskip

We denote by $\widehat{G}'$ the unitary dual of $G'$, and by $\widehat{G}'_d$ the subset of classes of square integrable  irreducible unitary representations. The elements of $\widehat{G}'_d$ are called discrete series representations of $G'$ and $\widehat{G}'_d$ contains the holomorphic ones~: $\Gholp(z)\croc\widehat{G}'_d$.

Since the moment map $\Phi^{G'}_{G\cdot\lambda}$ is proper, we know, thanks to the work of  T. Kobayashi \cite{Toshi-inventiones94} and Duflo-Vargas \cite{Duflo-Vargas2007}, that the unitary representation $V^G_\lambda$ is discretely admissible relatively to $G'$. It means that we have 
an Hilbertian direct sum 
$$
V^G_\lambda\vert_{G'}=\bigoplus_{\Pi\in \widehat{G}'_d} m_\lambda(\Pi) \ \Pi
$$
where the multiplicities $m_\lambda(\Pi)$ are finite. In fact, we can be more precise.

\begin{prop}Let $\lambda\in \Ghol(z)$. The multiplicty $m_\lambda(\Pi)$ is non-zero only if $\Pi= V^{G'}_\mu$ for some 
$\mu\in \Gholp(z)$. 
This means that we have 
$$
V^G_\lambda\vert_{G'}=\bigoplus_{\mu\in \Gholp(z)} m_\lambda(\mu) \ V^{G'}_\mu,
$$
with $m_\lambda(\mu) $ finite for any $\mu$.
\end{prop}

\begin{proof} Recall the parametrization of $\widehat{G}'_d$ given by Harish-Chandra. 
Let $(\wedge')^*\subset (\tgot')^*$ be the weight lattice. Let $\Rgot(\ggot',\tgot')^+$ be a choice of positive roots and let $\rho'$ be half the sum of its elements. The set $(\wedge')^* + \rho'$ does not depend of the choice of $\Rgot(\ggot',\tgot')^+$: we denote
it by $(\wedge')^*_\rho$. Let $(\tgot')^*_+\subset (\tgot')^*$ be the Weyl chamber corresponding to the choice of a set 
$\Rgot(\kgot',\tgot')^+$ of positive roots. 

The discrete series representations of $G'$ are parametrized by 
$$
\widehat{G}'_d:=\left\{\mu\in (\tgot')^*, \ \ggot' -{\rm regular}\right\}\cap (\wedge')^*_\rho\cap (\tgot')^*_+.
$$
At $\mu\in\widehat{G}'_d$, Harish-Chandra associates a square integrable unitary representation $\Pi^{G'}_\mu$ : $\mu$ is the {\em Harish-Chandra parameter}, and 
$$
\mu_{B}:=\mu-\rho'_c+\rho'_n(\mu)
$$
is the corresponding {\em Blattner parameter}. Here $\rho'_n(\mu)$ is associated to 
$\Rgot'_n(\mu)$. It is a classical fact that $\mu_{\small B}\in (\wedge')^*\cap (\tgot')^*_+$, and that the representation 
$V_{\mu_B}^{K'}$ occurs in $\Pi^{G'}_\mu\vert_{K'}$ with multiplicity one : $\mu_B$ is the minimal $K'$-types of $\Pi^{G'}_\mu$ in the sense of Vogan. Moreover the map $\mu\mapsto \mu_B$ induces a bijection between $\widehat{G}'_d\cap \Cholp(z)$ and $\Gholp(z)$ and we have  $\Pi^{G'}_\mu=V^{G'}_{\mu_B}$.

Let $\mu \in\widehat{G}'_d$ such that $[\Pi^{G'}_\mu : V^G_\lambda\vert_{G'}]\neq 0$. The Proposition will be proved if we check that 
$\mu\in \Cholp(z)$. Since $[V^{K'}_{\mu_B} : \Pi^{G'}_\mu\vert_{K'}]=1$, we have $[V^{K'}_{\mu_B} : V^G_\lambda\vert_{K'}]\neq 0$. 
Thanks to Proposition \ref{prop:restriction-G-K-prime}, we have then $\mu_B\in \cholp(z)$ : $\mu_B=2\rho'_n(z) +\xi$ with $\xi\in \Cholp(z)$. 
Hence 
$$
\mu = (\rho'_n(z) +\rho'_c)  + (\rho'_n(z)-\rho'_n(\mu))+\xi.
$$
The term $\rho':=\rho_n(z) +\rho_c'$ is associated to the choice of positive roots $\Rgot(\ggot',\tgot')^+:=\Rgot(\kgot',\tgot')^+\cup \Rgot'_n(z)$. Thus we have $(\rho',\alpha)>0$ for any $\alpha\in \Rgot'_n(z)$. 

The term $\rho'_n(z)-\rho'_n(\mu)$ is equal to the sum 
$$
\sum_{\stackrel{ \langle\alpha ,z\rangle >0}{(\alpha,\mu)<0}}\alpha
$$
and then $(\rho'_n(z)-\rho'_n(\mu),\alpha)\geq 0$ for any $\alpha\in \Rgot'_n(z)$. Finally we have proved that 
$$
(\mu,\alpha) = \underbrace{(\rho'_n(z) +\rho'_c,\alpha)}_{>0}  + \underbrace{(\rho'_n(z)-\rho'_n(\mu),\alpha)}_{\geq 0}
+\underbrace{(\xi,\alpha)}_{\geq 0}
$$
is positive for any $\alpha\in \Rgot'_n(z)$, thus $\mu \in\Cholp(z)$.
 \end{proof}
 
 \subsection{Jakobsen-Vergne's formula}

The aim of this section is to give a direct proof of the following result of Jakobsen-Vergne \cite{Jakobsen-Vergne}. 

\begin{theo}[Jakobsen-Vergne]\label{theo:Jakobsen-Vergne}
The multiplicity $m_\lambda(\mu)$ is equal to the multiplicity of the representation of $V^{K'}_\mu$ in 
$S^\bullet(\pgot/\pgot')\otimes V^{K}_\lambda\vert_{K'} $.
\end{theo}

Let us denote 
\begin{equation}\label{eq:Rfor-G-z}
\Rfor(G,z)
\end{equation}
the $\Zbb$-module formed by the infinite sum $\sum_{\lambda\in \Ghol(z)} m_\lambda \ V^{G}_\lambda$ with  
$m_\lambda\in\Zbb$. Similarly, we define $\Rfor(K,z)\subset \Rfor(K)$ as the sub-module formed by the infinite sum $\sum_{\mu\in \wedge^*_+} n_\mu \ V^{K}_\mu$ where $n_\mu\in\Zbb$ is non-zero only if $\mu\in\Khol(z)$. 

We have the following basic result.

\begin{lem}\label{lem:restriction-G-K} $\bullet$ The restriction to $K$ defines a morphism 
\begin{equation}\label{eq:restriction-K}
\mathbf{r}_{K,G} : \Rfor(G,z)\to\Rfor(K,z)
\end{equation}
 that is injective.
 
 $\bullet$ The product by $S^\bullet(\pgot)$ defines a map from $\Rfor(K,z)$ into itself.
\end{lem}

\begin{proof} Let us prove the first point. Thanks to Corollary \ref{cor:K-multiplicity}, we have $V^G_\lambda=\sum_{\mu\in \wedge^*_+} Q((G\cdot\lambda)_\mu) \ V^{K}_\mu$. Then 
\begin{eqnarray*}
\mathbf{r}_{K,G}\Big(\sum_{\lambda\in \Ghol(z)} m_\lambda \ V^{G}_\lambda\Big)&:=&
\sum_{\lambda\in \Ghol(z)} m_\lambda \ V^{G}_\lambda\vert_K\\
&=&\sum_{\mu\in \wedge^*_+}  \Big(\sum_{\lambda\in \Ghol(z)} m_\lambda Q((G\cdot\lambda)_\mu) \Big) V^{K}_\mu.
\end{eqnarray*}

We know that $Q((G\cdot\lambda)_\mu)\neq 0$ only if $\|\lambda\|\leq \|\mu\|$ and $\mu\in\Ccal(z)$. Hence the 
sum $n_\mu:=\sum_{\lambda\in \Ghol(z)} m_\lambda Q((G\cdot\lambda)_\mu)$ has a finite number of non-zero term and 
$n_\mu\neq 0$ only if $\mu\in\Chol(z)$.

Let $A=\sum_{\lambda\in \Ghol(z)} m_\lambda \ V^{G}_\lambda$ be a non-zero element in $\Rfor(G,z)$. Let 
$\lambda_A\in \Ghol(z)$ such that $\|\lambda_A\|$ is minimal among the set $\{\|\lambda\|\ \vert\ m_\lambda\neq 0\}$.
Let $\mathbf{r}_{K,G}(A)=\sum_{\mu} n_\mu \ V^{K}_\mu$. Then
$$
n_{\lambda_A}:=m_{\lambda_A} +\sum_{\lambda\neq \lambda_A} m_\lambda Q((G\cdot\lambda)_{\lambda_A}).
$$
But $m_\lambda=0$ if  $\|\lambda\|<\|\lambda_A\|$ and $Q((G\cdot\lambda)_{\lambda_A})=0$ if $\lambda\neq \lambda_A$ and 
$\|\lambda\|\geq\|\lambda_A\|$ (see second point of Corollary \ref{cor:K-multiplicity}). We have checked that $n_{\lambda_A}=m_{\lambda_A}\neq 0$ and then 
$\mathbf{r}_{K,G}(A)\neq 0$.

Let us check the second point. Let $A= \sum_{\mu\in \wedge^*_+} n_\mu V^K_\mu\in \Rfor(K,z)$. Then 
\begin{eqnarray*}
A\otimes S^\bullet(\pgot)&=& \sum_{\mu\in \Ccal(z)} n_\mu V^K_\mu\otimes S^\bullet(\pgot)\\
&=& \sum_{\theta} \Big(\sum_{\mu\in \Ccal(z)}n_\mu Q((G\cdot\mu)_\theta)\Big) V^K_\theta.
\end{eqnarray*}
Like before, the term $Q((G\cdot\mu)_\theta)$ is non-zero only if $\|\mu\|\leq \|\theta\|$ and $\theta\in\Ccal(z)$. Hence the 
sum $\sum_{\mu\in \Ccal(z)}n_\mu Q((G\cdot\mu)_\theta)$ has a finite number of non-zero term and 
is non-zero only if $\theta\in\Ccal(z)$.
\end{proof}

\medskip

Let us consider the similar morphism $\mathbf{r}_{K',G'} : \Rfor(G',z)\to\Rfor(K',z)$ for the reductive subgroup $G'$. We consider the following elements of $\Rfor(G',z)$: 
\begin{eqnarray*}
V^G_\lambda\vert_{G'}&=&\sum_{\mu\in \Gholp(z)} m_\lambda(\mu) \ V^{G'}_\mu,\quad \mathrm{and}\\
\delta &:=& \sum_{\mu\in \Gholp(z)} n_\lambda(\mu) \ V^{G'}_\mu,
\end{eqnarray*}
where $n_\lambda(\mu):=[V^{K'}_\mu : S^\bullet(\pgot/\pgot')\otimes V^{K}_\lambda\vert_{K'}]$. Theorem \ref{theo:Jakobsen-Vergne} 
will be proved if we check that $\mathbf{r}_{K',G'}(V^G_\lambda\vert_{G'})=\mathbf{r}_{K',G'}(\delta)$. But 
$\mathbf{r}_{K',G'}(V^G_\lambda\vert_{G'})= V^G_\lambda\vert_{K'}=S^\bullet(\pgot)\otimes V^{K}_\lambda\vert_{K'}$, and 
\begin{eqnarray*}
\mathbf{r}_{K',G'}(\delta)&=&\sum_{\mu\in \Gholp(z)} n_\lambda(\mu) \ V^{G'}_\mu\vert_{K'}\\
&=&S^\bullet(\pgot')\otimes \Big(\sum_{\mu\in \Gholp(z)} n_\lambda(\mu) \  V^{K'}_\mu\Big) \qquad [1]\\
&=&S^\bullet(\pgot')\otimes \Big(\sum_{\mu\in \Gholp(z)} [V^{K'}_\mu : S^\bullet(\pgot/\pgot')\otimes V^{K}_\lambda\vert_{K'} ] 
\ V^{K'}_\mu\Big)\\
&=& S^\bullet(\pgot)\otimes V^{K}_\lambda\vert_{K'}
\end{eqnarray*}

The second point of Lemma \ref{lem:restriction-G-K} insures that the product in $[1]$ is well-defined. We need to explain the last equality. Note that $[V^{K'}_\mu,\, S^\bullet(\pgot/\pgot')\otimes V^{K}_\lambda\vert_{K'} ]\neq 0$ implies 
$[V^{K'}_\mu,\, S^\bullet(\pgot)\otimes V^{K}_\lambda\vert_{K'} ]\neq 0$ and then $\mu\in \Gholp(z)$ (see Proposition \ref{prop:restriction-G-K-prime}). This insures that the sum 
$$
\sum_{\mu\in \Gholp(z)} [V^{K'}_\mu : S^\bullet(\pgot/\pgot')\otimes V^{K}_\lambda\vert_{K'}]\ V^{K'}_\mu\quad \in\ 
\Rfor(K',z)
$$
is equal to $S^\bullet(\pgot/\pgot')\otimes V^{K}_\lambda\vert_{K'}$.

\subsection{Formal geometric quantization of $G$-actions}\label{sec:formal-G}

In this section we consider the Hamiltonian action of a connected real reductive Lie group $G$ on a symplectic manifold $(M,\Omega_M)$. 
We suppose that the action of $G$ on $M$ is {\em proper} and that the moment map $\Phi^G_M:M\to\ggot^*$ is {\em proper}. We know that we have 
a global slice $Y\subset M$ such that 
$$
M\simeq G\times_K Y,
$$
 and that the $G$-orbits in the image of $\Phi^G_M$ are parametrized by the Kirwan polytope $\Delta_{K}(Y)$.

Let us suppose the existence of a $G$-equivariant pre-quantum line bundle $L_M\to M$. Note that $L_M$ is completely determined by its restriction $L_Y\to Y$ to the sub-manifold $Y$: here $L_Y$ is a $K$-equivariant pre-quantum line bundle over $(Y,\Omega_Y)$. For any dominant weight $\mu$, we see that the reduce space  
$$
M_{\mu,G}:=(\Phi^G_M)^{-1}(G\cdot\mu)/G
$$
coincides with $Y_{\mu,K}:=(\Phi^K_M)^{-1}(K\cdot\mu)/K$. Hence its quantization 
$$
\Qcal(M_{\mu,G}):=\Qcal(Y_{\mu,K})\in\Zbb
$$
 is well-defined (see Section \ref{subsec:formal-quantization}).

We suppose also that $G$ satisfies (\ref{condition-hol}), and we fix a complex structure $\mathrm{ad}(z)$ on $\pgot$. 
Let $\chol(z)\subset \tgot^*$ be the corresponding cone. 

\begin{lem}\label{lem-phi-z-positif}
Let Let $(M,\Omega_M,\Phi^G_M)$ be a  Hamiltonian manifold. Suppose that the image of $\Phi^G_M$ is contained in $G\cdot\chol(z)\subset\ggot_{se}^*$. 
Then :
\begin{enumerate}
\item the Kirwan polytopes $\Delta_K(Y)\subset\Delta_K(M)$ are contained in $\chol(z)$,
\item the functions $\langle\Phi^K_Y,z\rangle$ and $\langle\Phi^K_M,z\rangle$ take strictly positive values.
\end{enumerate}
\end{lem}

\begin{proof} The Kirwan polytope $\Delta_K(M)= \pi_{\kgot,\ggot}\left(\mathrm{Image}(\Phi^G_M)\right)\cap \tgot^*_+$ is contained in
$$
\bigcup_{\lambda\in\chol(z)}\Delta_K(G\cdot\lambda)\subset\chol(z),
$$
where the last inclusion is a consequence of point $(a)$ in Proposition \ref{prop:Delta-K-Lambda}. The first point is proved. Hence we get, thanks to (\ref{eq:inclusion-cone}), 
the following relations 
$$
\mathrm{Image}(\langle\Phi^K_Y,z\rangle)\subset \mathrm{Image}(\langle\Phi^K_M,z\rangle)\subset \langle\Delta_K(M),z\rangle\subset \langle\chol(z),z\rangle\subset [c, +\infty[
$$
with $c=\langle\rho(z),z\rangle=\frac{1}{2}\dim\pgot$.
\end{proof}

\medskip

We have the following notion of formal geometric quantization 
that extends the case of compact Lie group actions.

\begin{defi} \label{def:formal-G}
Let $(M,\Omega_M,\Phi^G_M)$ be a pre-quantized Hamiltonian manifold, such that the moment map 
$\Phi^G_M$ is a {\em proper} map from $M$ into $G\cdot\chol(z)$.
Then we define the formal geometric quantization of $M$ as the following element of $\Rfor(G,z)$:
$$
\qfor_G(M):=\sum_{\mu\in \Ghol(z)} \Qcal(M_{\mu,G})\ V^G_\mu.
$$
\end{defi}

\medskip

Let $\mathbf{r}_{K,G}: \Rfor(G,z)\to \Rfor(K,z)$ be the restriction morphism defined in Lemma \ref{lem:restriction-G-K}. Recall that in the setting of Definition \ref{def:formal-G}, the moment map $\Phi^K_M$ is proper (see Theorem \ref{theo:properGK}). Then the formal geometric quantization of $M$ relatively to the $K$-action is well-defined : $\qfor_K(M)\in \Rfor(K)$.

We have proved in Theorem \ref{theo:properGK}, that the sets of critical points of the function $\|\Phi_M^G\|^2, \|\Phi_M^K\|^2$ and 
$\|\Phi_Y^K\|^2$ are equal.  We will be interested by one of the following hypothesis:

\begin{assu}\label{assumption}
\begin{itemize}
\item {\bf A1} The set $\Cr(\|\Phi_M^G\|^2)$ is compact.
\item {\bf A2} The map $\langle\Phi_M^G,z\rangle:M\to\Rbb$ is proper.
\end{itemize}
\end{assu}

\medskip

In the following Lemma, we exhibit examples where the Assumptions {\bf A1} or {\bf A2} are satisfied.

\begin{lem}$\bullet$ Suppose that we are in the algebraic setting: the manifold $M$ is real algebraic and the map $\Phi^G_M$ is a proper algebraic map. Then  $\Cr(\|\Phi_M^G\|^2)$ is compact.

$\bullet$ Suppose that the Lie algebra $\ggot$ is simple. Then, in the context of Definition \ref{def:formal-G}, the map $\langle\Phi_M^G,z\rangle:M\to\Rbb$ is proper.
\end{lem}

\begin{proof} Let us prove the first point. The map $\varphi:=\|\Phi_M^G\|^2:M\to\Rbb$ is a real algebraic map on a real algebraic manifold. Thus the set $\Cr(\varphi)$ is an algebraic variety, and by a standard Theorem of Whitney, it as a finite number of connected components $C_1,\cdots,C_p$. Each $C_i$ is contained in $\varphi^{-1}(\varphi(C_i))$ which is compact since $\varphi$ is proper. The proof is completed.

For the  second point we use the result of Proposition \ref{prop:map=propre}, and the facts that, since $\ggot$ is simple, $[\pgot,\pgot]=\kgot$ 
and the center $\cgot_\kgot$ of $\kgot$ is reduced to $\Rbb z$.

The function $\langle\Phi_M^G,z\rangle$, which is the moment map for the $S^1$-action, is proper if and only if $\As(\Delta_K(M))\cap (\Rbb z)^\perp=\{0\}$. Since $\Delta_K(M)\subset\chol(z)$ (see Lemma \ref{lem-phi-z-positif}), it is sufficient to prove that $\Chol(z) \cap (\Rbb z)^\perp=\{0\}$. Let $\xi\in \Chol(z)$. We have 
$\langle\xi, z\rangle=-b(\tilde{\xi}, z)= 2\sum_{\beta\in\Rgot_n(z)}\langle\beta, \tilde{\xi}\rangle$ with $\langle\beta, \tilde{\xi}\rangle=
(\beta,\xi)\geq 0$. If $\langle\xi, z\rangle=0$, we must have $(\beta,\xi) =0, \forall \beta\in \Rgot_n(z)$ or equivalently $[\tilde{\xi},\pgot]=0$. 
Then $\tilde{\xi}$ commutes with all elements in $[\pgot,\pgot]=\kgot$, i.e. $\tilde{\xi}\in\cgot_\kgot=\Rbb z$.  Finally, we have proved that $\xi\in (\Rbb z)^\perp$ and $\tilde{\xi}\in \Rbb z$, hence $\xi=0$.
\end{proof}

\medskip

We can now state the main result of this section.

\medskip

\begin{theo}\label{theo:formal-G-K}
If Assumptions {\bf A1} or {\bf A2} are satisfied, we have the following relation
$$
\mathbf{r}_{K,G}\Big(\qfor_G(M)\Big)=\qfor_K(M).
$$ 
\end{theo}

\medskip

\begin{proof}
We have 
\begin{eqnarray*}
\mathbf{r}_{K,G}\Big(\qfor_G(M)\Big)&=&\sum_{\mu\in \Ghol(z)} \Qcal(M_{\mu,G})\ V^G_\mu\vert_K\\
&=& \Big(\sum_{\mu\in \wedge^*_+\cap \chol(z)} \Qcal(Y_{\mu})\ V^K_\mu\Big)\otimes S^\bullet(\pgot)\qquad [1]\\
&=& \qfor_K(Y) \otimes S^\bullet(\pgot)\qquad [2].
\end{eqnarray*}

Note that the product in $[1]$ and $[2]$ are well defined thanks to Lemma \ref{lem:restriction-G-K}. In $[2]$ we use the fact that 
$\qfor_K(Y)=\sum_{\mu\in \wedge^*_+\cap \chol(z)} \Qcal(Y_{\mu})V^K_\mu$ since by hypothesis $\Delta_K(Y)\subset \chol(z)$. 
So Theorem \ref{theo:formal-G-K} follows from the following equality
\begin{equation}\label{eq:formal-G-K}
\qfor_K\left(G\times_K Y\right)=\qfor_K(Y) \otimes S^\bullet(\pgot),
\end{equation}
that will be proved in Sections \ref{sec:prof-theorem-formal-G-K-1} and \ref{sec:prof-theorem-formal-G-K-2}.
\end{proof}

\medskip

We consider now a connected reductive subgroup $G'\subset G$ such that $z\in \ggot'$. The coadjoint orbit 
$G\cdot\lambda$ is pre-quantized when $\lambda\in \Ghol(z)$ and we have obviously $\qfor_G(G\cdot\lambda)=V^G_\lambda$. 
The moment map $\Phi^{G'}_{G\cdot\lambda} :G\cdot \lambda\to(\ggot')^*$ relative to the $G'$-action on $G\cdot\lambda$ is proper. In fact we have more : the map $\langle\Phi^{G'}_{G\cdot\lambda},z\rangle :G\cdot \lambda\to\Rbb$ is proper, thus Assumption {\bf A2} holds. 

We are interested in the compact reduced spaces
$$
(G\cdot\lambda)_{\mu,G'}:=(\Phi^{G'}_{G\cdot\lambda})^{-1}(G'\cdot\mu)/G',
$$
for $\mu'\in \Ghol(z)$. We are now able to prove the following

\begin{theo}\label{theo:QR=0-non-compact} Let $\lambda\in \Ghol(z)$. Then we have the following relation
$$
V^G_\lambda\vert_{G'}=\qfor_{G'}(G\cdot\lambda)
$$
in $\Rfor(G',z)$. It means that for any $\mu\in \Gholp(z)$, the multiplicity of the representation $V^{G'}_\mu$ in 
the restriction $V^G_\lambda\vert_{G'}$ is equal to the geometric quantization
$$
\Qcal\Big((G\cdot\lambda)_{\mu,G'}\Big)\in \Zbb
$$ 
of the (compact) reduced space $(G\cdot\lambda)_{\mu,G'}$.
\end{theo}

\begin{proof} Since the restriction morphism $\mathbf{r}_{K',G'}:\Rfor(G',z)\to\Rfor(K',z)$ is injective (see Lemma \ref{lem:restriction-G-K}) it 
suffices to prove that 
\begin{equation}\label{eq:preuve-QR=0}
\mathbf{r}_{K',G'}\Big(V^G_\lambda\vert_{G'}\Big)= \mathbf{r}_{K',G'}\Big(\qfor_{G'}(G\cdot\lambda)\Big).
\end{equation} 
But the left hand side of (\ref{eq:preuve-QR=0}) is equal to the restriction $V^G_\lambda\vert_{K'}$, while the right hand side is equal to 
$\qfor_{K'}(G\cdot\lambda)$ thanks to Theorem \ref{theo:formal-G-K}. Theorem \ref{theo:Q-Phi-G-lambda} tells us that 
$\qfor_{K}(G\cdot\lambda)= V^G_\lambda\vert_{K}$ and the functoriality of the quantization process $\qfor$ 
(see Theorem \ref{theo:pep-formal}) insures that the restriction $V^G_\lambda\vert_{K'}=\qfor_{K}(G\cdot\lambda)\vert_{K'}$ is equal to 
$\qfor_{K'}(G\cdot\lambda)$.
\end{proof}

\bigskip

We can finish this section by extending the functoriality of the quantization process $\qfor$ relatively to the restrictions. 

\begin{defi}
An element $m:=\sum_{\lambda\in \Ghol(z)} m_\lambda \ V^{G}_\lambda\in \Rfor(G,z)$ is admissible relatively to $G'$ if 
the projection $\pi_{\ggot',\ggot}$ is proper when restricted to the subset $G\cdot \mathrm{Support}(m)\subset G\cdot\chol(z)$, where  $\mathrm{Support}(m)=\{\lambda\, \vert\, m_\lambda\neq 0\}\subset \Ghol(z)$.

The same definition holds for the couple $(K',K)$.
\end{defi}
 
 When $m\in \Rfor(G,z)$ is $G'$-admissible, we can define its restriction 
\begin{eqnarray*}
\mathbf{r}_{G',G}(m)&:=& \sum_{\lambda\in \Ghol(z)} m_\lambda  V^{G}_\mu\vert_{G'}\\
&=&\sum_{\mu\in \Gholp(z)}\left(\sum_{\lambda\in \Ghol(z)} m_\lambda Q((G\cdot\lambda)_{\mu,G'})\right) 
 V^{G'}_\mu\in \Rfor(G',z).
\end{eqnarray*}
Note that for any $\mu\in\Gholp(z)$ the sum $\sum_{\lambda\in \Ghol(z)} m_\lambda Q((G\cdot\lambda)_{\mu,G'})$ has only a finite number 
of non-zero term. Similarly, when $n=\sum_{b\in \Khol(z)} n_b \ V^{K}_b\in \Rfor(K,z)$ is $K'$-admissible, we can define its restriction 
\begin{eqnarray*}
\mathbf{r}_{K',K}(n)&:=& \sum_{b\in \Khol(z)} n_b V^{K}_b\vert_{K'}\\
&=&\sum_{\mu\in \Kholp(z)}\left(\sum_{b\in \Khol(z)} n_b Q((K\cdot b)_{a,K'})\right) 
 V^{K'}_a\in \Rfor(K',z).
\end{eqnarray*}

We will used the following Lemma that will be proved in the Appendix.

\begin{lem}\label{lem:restriction-admissible}
Let $m\in \Rfor(G,z)$ that is $G'$-admissible. Then $\mathbf{r}_{K,G}(m)\in \Rfor(K,z)$ is $K'$-admissible and the following  relation
$$
\mathbf{r}_{K',K}\circ\mathbf{r}_{K,G}(m)= \mathbf{r}_{K',G'}\circ\mathbf{r}_{G',G}(m)
$$
holds in $\Rfor(K',z)$.
\end{lem}

\medskip

We finish this section with the following 

\begin{theo}Let $(M,\Omega_M,\Phi^G_M)$ be a pre-quantized Hamiltonian manifold. Suppose that $\mathrm{Image}(\Phi^G_M)\subset G\cdot\Chol(z)$, and that the map $\langle\Phi^G_M,z\rangle$ is {\em proper}. Let $G'$ be a reductive subgroup such that $z\in\ggot'$.  Then: 

$\bullet$ The map $\Phi^{G'}_M$ is proper and $\mathrm{Image}(\Phi^{G'}_M)\subset G'\cdot\cholp(z)$. So $\qfor_{G'}(M)\in \Rfor(G',z)$ is well defined.

$\bullet$ The element $\qfor_{G}(M)\in \Rfor(G,z)$ is $G'$-admissible and we have 
$$
\mathbf{r}_{G',G}\left(\qfor_G(M)\right)=\qfor_{G'}(M).
$$
\end{theo}

%
%

\begin{proof} The map $\Phi^{G'}_M$ is proper since $\langle\Phi^G_M,z\rangle$ is proper. The point concerning the image of $\Phi^{G'}_M$ is a consequence of point $(d)$ in Proposition \ref{prop:cone-projection}. Let $m=\qfor_{G}(M)\in \Rfor(G,z)$. 
By definition $\mathrm{Support}(m)$ is contained in $\mathrm{Image}(\Phi^G_M)$, and then 
$G\cdot \mathrm{Support}(m)\subset \mathrm{Image}(\Phi^G_M)$. Since the moment map $\Phi^{G'}_M$ is proper, we know that 
the projection $\pi_{\ggot',\ggot}$ is proper when restricted to $\mathrm{Image}(\Phi^G_M)$. This implies that $m\in\Rfor(G,z)$ is $G'$-admissible.

We have then
\begin{eqnarray*}
\mathbf{r}_{K',G'}\circ\mathbf{r}_{G',G}\left(\qfor_G(M)\right)&=&\mathbf{r}_{K',K}\circ\mathbf{r}_{K,G}\left(\qfor_G(M)\right)\qquad [1] \\
&=& \mathbf{r}_{K',K}\left(\qfor_K(M)\right) \qquad [2] \\
&=& \qfor_{K'}(M).\qquad [3]
\end{eqnarray*}
Here $[1]$ follows from Lemma \ref{lem:restriction-admissible}, $[2]$ follows from Theorem \ref{theo:formal-G-K} and 
$[3]$ is the consequence of Theorem \ref{theo:pep-formal}. We have checked that 
$$
\mathbf{r}_{K',G'}\circ\mathbf{r}_{G',G}\left(\qfor_G(M)\right)= \qfor_{K'}(M)= \mathbf{r}_{K',G'}\left(\qfor_{G'}(M)\right).
$$
Since the map $\mathbf{r}_{K',G'}$ is injective, it follows that $\mathbf{r}_{G',G}\left(\qfor_G(M)\right)=\qfor_{G'}(M)$.
\end{proof}

\subsection{Geometric quantization of the slice $Y$}\label{subsec:quantization-Y}

Let $\lambda\in\Gholp(z)$. Consider\footnote{In this section, we interchange the role of the groups $G$ and $G'$ in order to minimize the primes in the notation.} the coadjoint orbit $G'\cdot\lambda$ associated to the holomorphic discrete series representation $V^{G'}_\lambda$. Let $G$ be a reductive subgroup of $G'$ such that $z\in\ggot$. 
We know that we have a geometric decomposition
$$
G'\cdot\lambda= G\times_{K} Y
$$
where $Y\subset G'\cdot\lambda$ is a closed $K$-invariant symplectic sub-manifold. 

We have two ways of computing the multiplicity of $m_\lambda(\mu)$ of $V^{G}_\mu$ in $V^{G'}_\lambda$. First, after 
Jakobsen-Vergne, we know that 
$$
m_\lambda(\mu)= \left[V^{K}_\mu : S^\bullet(\pgot'/\pgot)\otimes V^{K'}_\lambda\vert_{K} \right],
$$
and Theorem \ref{theo:QR=0-non-compact} tells us also that 
$$
m_\lambda(\mu)=\Qcal((G'\cdot\lambda)_{\mu,G})=\Qcal(Y_{\mu,K}).
$$
We would like to understand a priori why $\Qcal(Y_{\mu,K})=[V^{K}_\mu : S^\bullet(\pgot'/\pgot)\otimes V^{K'}_\lambda\vert_{K} ]$ for any 
$\mu\in \Khol(z)$, or equivalently why we have the relation
\begin{equation}\label{eq:formal-quant-Y-group}
\qfor_{K}(Y)= S^\bullet(\pgot'/\pgot)\otimes V^{K'}_\lambda\vert_{K}.
\end{equation}

Note that Assumption {\bf A2} holds in this setting : the map $\langle\Phi^{G}_{G'\cdot\lambda},z\rangle$ is  proper.

\medskip

Let us consider  a more general situation. Let $(M,\Omega_M,\Phi^{G}_M)$ be a pre-quantized Hamiltonian $G$-manifold. We suppose the $G$-action proper, and that the moment map $\Phi^{G}_M$ takes values in $G\cdot\chol(z)$. We suppose furthermore that Assumption {\bf A2} holds. Let $Y\subset M$ be the symplectic slice. The aim of this section is to compute $\qfor_{K}(Y)$ in a way similar to (\ref{eq:formal-quant-Y-group}).

\medskip

Let $\Xcal$ be a connected component of $Y^z$. Let us fix a $K$-invariant almost complex structure on $\Xcal$ which is compatible with the symplectic structure. Let 
$$
\RR^{K}(\Xcal, -)
$$ 
be the corresponding Riemann-Roch character (see Section \ref{sec:RR}). Recall that, if $L_{\Xcal}$ denotes the restriction of the Kostant-Souriau line bundle $L_M$ on $\Xcal$, we have $\Qcal_{K}(\Xcal)=\RR^{K}(\Xcal, L_{\Xcal})$.

Let $\Ncal_\Xcal\to \Xcal$ be the normal bundle of $\Xcal$ in $Y$ : it inherits a complex structure $J_\Xcal$ and a linear endomorphism  
$\Lcal(z)$ on the fibres. We have a decomposition $\Ncal_{\Xcal}=\sum_{a\in\Rbb}\Ncal^a_{\Xcal}$ where $\Ncal^a_{\Xcal}
=\{v\in \Ncal_{\Xcal}\, \vert\, \Lcal(z)v= a J_{\Xcal}(v)\}$ is a sub-bundle of $\Ncal_{\Xcal}$. We define the vector bundle $\Ncal^{\pm,z}_{\Xcal}:=\sum_{\pm a>0}\Ncal_{\Xcal}$ and 
$$
\vert\Ncal\vert^{z}=\Ncal^{+,z}_{\Xcal}\oplus \overline{\Ncal^{-,z}_{\Xcal}}.
$$

\begin{theo}\label{theo:qfor-Y}
We have the following equality in $\Rfor(K)$:
$$
\qfor_{K}(Y)=\sum_{\Xcal}(-1)^{r_\Xcal}\RR^{K}\left(\Xcal, L_{\Xcal}\otimes \det(\Ncal^{+,z}_{\Xcal})\otimes S^\bullet(\vert\Ncal_{\Xcal}\vert^{z})\right),
$$
where $r_\Xcal$ is the complex rank of $\Ncal^{+,z}_{\Xcal}$.
\end{theo}

The proof will be given in Section \ref{sec:proof-theo-qfor-Y}. 

Let us explain how  the formulas of Jakobsen-Vergne can be recover with Theorem \ref{theo:qfor-Y}. When $M=G'\cdot\lambda$, the sub-manifolds $Y^z$ and $M^z$ are both equal to $K'\cdot\lambda$. The restriction 
of the Kostant-Souriau line bundle $L_M\to M$ on $Y^z$ is $[\Cbb_\lambda]:=K'\times_{K'_\lambda}\Cbb_\lambda\to K'\cdot\lambda$. 
Relation (\ref{rel=3}) tells us that the normal bundle $\Ncal_1$ of $Y$ in $M$ is equal to the trivial bundle $\pgot\times Y$, and the normal bundle $\Ncal_2$ of $Y^z$ in $M$ is equal to $Y^z\times \pgot'$. Hence the normal bundle of $Y^z$ in $Y$ is 
$$
\Ncal=\Ncal_2/(\Ncal_1\vert_{Y^z})=Y^z\times (\pgot'/\pgot).
$$
We check that $\Ncal^{+,z}=0$ : this is due to the fact that the function $\langle\Phi^{G}_{G'\cdot\lambda},z\rangle$ takes its minimal value on $Y^z=K'\cdot\lambda$ (see Lemma 7.3 in \cite{pep-RR}). So $\vert\Ncal\vert^{z}=\overline{\Ncal}$ is the trivial complex bundle with fiber $(\pgot'/\pgot,\mathrm{ad}(z))$. Theorem \ref{theo:qfor-Y} gives 
\begin{eqnarray*}
\qfor_{K}(Y)&=&\RR^{K}\left(K'\cdot\lambda, [\Cbb_\lambda]\otimes S^\bullet(\pgot'/\pgot)\right)\\
&=&\RR^{K'}\left(K'\cdot\lambda, [\Cbb_\lambda]\right)\vert_{K}\otimes S^\bullet(\pgot'/\pgot)\\
&=& V_{\lambda}^{K'}\vert_{K}\otimes S^\bullet(\pgot'/\pgot).\qquad [1]
\end{eqnarray*}
In $[1]$, we use that $\RR^{K'}(K'\cdot\lambda, [\Cbb_\lambda])=V_{\lambda}^{K'}$ thanks to the Borel-Weil theorem.

\section{Transversally elliptic operators}\label{sec:transversally}

The aim of this section is to give a proof of Theorem \ref{theo:formal-G-K} and \ref{theo:qfor-Y}.
In the first section, we briefly introduce the material we need from the theory of transversally elliptic operator. And in Section \ref{sec:Qcal-Phi}
 we recall the definition of the geometric quantization process $\Qcal^\Phi$. In the rest of this paper, $K$ will denoted a connected compact Lie group.

\subsection{Transversally elliptic operators}
Here we give the basic definitions from the theory of transversally
elliptic symbols (or operators) defined by Atiyah-Singer in
\cite{Atiyah74}. For an axiomatic treatment of the index morphism
see Berline-Vergne \cite{B-V.inventiones.96.1,B-V.inventiones.96.2} and 
Paradan-Vergne \cite{pep-vergne:bismut}. For a short introduction see \cite{pep-RR}.

Let $\Xcal$ be a {\em compact} $K$-manifold. Let $p:\T
\Xcal\to \Xcal$ be the projection, and let $(-,-)_\Xcal$ be a
$K$-invariant Riemannian metric. If $E^{0},E^{1}$ are
$K$-equivariant complex vector bundles over $\Xcal$, a
$K$-equivariant morphism 
$$
\sigma \in \Gamma(\T\Xcal,\hom(p^{*}E^{0},p^{*}E^{1}))
$$ 
is called a {\em symbol} on $\Xcal$. The
subset of all $(x,v)\in \T \Xcal$ where\footnote{The map $\sigma(x,v)$ will be also denote 
$\sigma\vert_x(v)$} $\sigma(x,v): E^{0}_{x}\to
E^{1}_{x}$ is not invertible is called the {\em characteristic set}
of $\sigma$, and is denoted by $\Char(\sigma)$.

In the following, the product of a symbol  $\sigma$ by a complex vector bundle $F\to M$, is the symbol 
$\sigma\otimes F$ defined by $\sigma\otimes F(x,v)=\sigma(x,v)\otimes {\rm Id}_{F_x}$ from 
$E^{0}_x\otimes F_x$ to $E^{1}_x\otimes F_x$. Note that $\Char(\sigma\otimes F)=\Char(\sigma)$.

Let $\T_{K}\Xcal$ be the following subset of $\T \Xcal$ :
$$
   \T_{K}\Xcal\ = \left\{(x,v)\in \T \Xcal,\ (v,X_{\Xcal}(x))_{_{\Xcal}}=0 \quad {\rm for\ all}\
   X\in\kgot \right\} .
$$

A symbol $\sigma$ is {\em elliptic} if $\sigma$ is invertible
outside a compact subset of $\T \Xcal$ (i.e. $\Char(\sigma)$ is
compact), and is $K$-{\em transversally elliptic} if the
restriction of $\sigma$ to $\T_{K}\Xcal$ is invertible outside a
compact subset of $\T_{K}\Xcal$ (i.e. $\Char(\sigma)\cap
\T_{K}\Xcal$ is compact). An elliptic symbol $\sigma$ defines an
element in the equivariant $\Ko$-theory of $\T\Xcal$ with compact
support, which is denoted by $\Ko_{K}(\T \Xcal)$, and the
index of $\sigma$ is a virtual finite dimensional representation of
$K$, that we denote $\indice^K_{\Xcal}(\sigma)\in R(K)$
\cite{Atiyah-Segal68,Atiyah-Singer-1,Atiyah-Singer-2,Atiyah-Singer-3}.

A $K$-{\em transversally elliptic} symbol $\sigma$ defines an
element of $\Ko_{K}(\T_{K}\Xcal)$, and the index of
$\sigma$ is defined as a trace class virtual representation of $K$, that we still denote 
$\indice^K_\Xcal(\sigma)\in \Rfor(K)$ \cite{Atiyah74}.

Using the {\em excision property}, one can easily show that the
index map $\indice^K_\Ucal: \Ko_{K}(\T_{K}\Ucal)\to
\Rfor(K)$ is still defined when $\Ucal$ is a
$K$-invariant relatively compact open subset of a
$K$-manifold (see \cite{pep-RR}[section 3.1]).

\medskip

Suppose now that the group $K$ is equal to the product $K_1\times K_2$. 
An intermediate notion between the ``ellipticity'' and ``$K_1\times K_2$-transversal ellipticity'' is 
the ``$K_1$-transversal ellipticity''. When a $K_1\times K_2$-equivariant symbol $\sigma$ is 
{\em $K_1$-transversally elliptic}, its index $\indice^{K_1\times K_2}_\Xcal(\sigma)\in \Rfor(K_1\times K_2)$, 
viewed as a generalized function on $K_1\times K_2$, is {\em smooth} relatively to the variable in $K_2$ 
\cite{Atiyah74,B-V.inventiones.96.2,pep-vergne:bismut}. It implies that :

$\bullet$ $\indice^{K_1\times K_2}_\Xcal(\sigma)=\sum_{\lambda\in\what{K_1}} \theta_\lambda\otimes V_\lambda^{K_1}$ with $\theta_\lambda\in R(K_2)$,

$\bullet$ we can restrict $\indice^{K_1\times K_2}_\Xcal(\sigma)$ to the subgroup $K_1$ and 
\begin{equation}\label{eq:1-2-index}
\indice^{K_1\times K_2}_\Xcal(\sigma)\vert_{K_1}= \sum_{\lambda\in\what{K_1}} \dim(\theta_\lambda)\, V_\lambda^{K_1}=
\indice^{K_1}_\Xcal(\sigma).
\end{equation}
Here $\dim:R(K_2)\to \Zbb$ is the morphism induced by the restriction to $1\in K_2$.

\medskip

Let us recall the multiplicative property of the index map for the product
of manifolds that was proved by Atiyah-Singer in \cite{Atiyah74}.  Consider a compact Lie group $K_2$ acting on two
manifolds $\Xcal_1$ and $\Xcal_2$, and assume that another compact Lie group $K_1$ acts on $\Xcal_1$ 
commuting with the action of $K_2$. The external product of complexes on $\T\Xcal_1$ and $\T\Xcal_2$ induces
a multiplication (see \cite{Atiyah74}):
$$
\odot: \Ko_{K_1\times K_2}(\T_{K_1} \Xcal_1)\times \Ko_{K_2}(\T_{K_2} \Xcal_2)
\longrightarrow \Ko_{K_1\times K_2}(\T_{K_1\times K_2} (\Xcal_1\times \Xcal_2)).
$$

Let us recall the definition of this external product. For $k=1,2$, we consider equivariant
morphisms\footnote{In order to simplify the notation, we do not make the distinctions between
vector bundles on $\T\Xcal$ and on $\Xcal$.} $\sigma_k:\Ecal^+_k\to \Ecal_k^-$  on $\T\Xcal_k$. We consider the equivariant morphism on $\T(\Xcal_1\times \Xcal_2)$
$$
\sigma_1\odot\sigma_2: \Ecal_1^{+}\otimes \Ecal_2^{+}\oplus \Ecal_1^{-}\otimes \Ecal_2^{-}
\longrightarrow \Ecal_1^{-}\otimes \Ecal_2^{+} \oplus \Ecal_1^{+}\otimes \Ecal_2^{-}
$$
defined by
\begin{equation}\label{eq:produit.externe}
\sigma_{1}\odot \sigma_{2}=
\left(
\begin{array}{cc}
  \sigma_{1}\otimes {\rm Id} & -{\rm Id}\otimes \sigma_{2}^{*}\\
{\rm Id} \otimes \sigma_{2} & \sigma_{1}^{*}\otimes {\rm Id}
\end{array}
\right)\ .
\end{equation}

We see that the set $\Char(\sigma_{1}\odot \sigma_{2})\subset \T\Xcal_1\times\T\Xcal_2$ is equal to
$\Char(\sigma_{1})\times \Char(\sigma_{2})$. We suppose now that the morphisms $\sigma_k$ are respectively $K_k$-transversally elliptic. Since $\T_{K_1\times K_2}(\Xcal_1\times \Xcal_2)\neq \T_{K_1} \Xcal_1\times\T_{K_2} \Xcal_2$, the
morphism $\sigma_{1}\odot\sigma_{2}$ is not necessarily $K_1\times K_2$-transversally elliptic. Nevertheless, if
$\sigma_2$ is taken {\em almost homogeneous}, then the morphism $\sigma_{1}\odot \sigma_{2}$ is $K_1\times K_2$-transversally elliptic (see \cite{pep-vergne:bismut}). So the exterior product $a_1\odot a_2$ is the $\K$-theory class defined by $\sigma_1\odot\sigma_2$,  where $a_k=[\sigma_k]$ and $\sigma_2$ is taken almost homogeneous.

The following property is a useful tool (see \cite{Atiyah74}[Lecture 3] and \cite{pep-vergne:bismut}).

\begin{theo}[Multiplicative property] \label{theo:multiplicative-property}
For any $[\sigma_1]\in \Ko_{K_1\times K_2}(\T_{K_1} \Xcal_1)$ and
any $[\sigma_2]\in \Ko_{K_2}(\T_{K_2} \Xcal_2)$ we have
$$
\indice^{K_1\times K_2}_{\Xcal_1\times \Xcal_2}([\sigma_1]\odot[\sigma_2])
=\indice^{K_1\times K_2}_{\Xcal_1}([\sigma_1])\otimes\indice^{K_2}_{\Xcal_2}([\sigma_2]).
$$
\end{theo}

\subsection{Riemann-Roch character}\label{sec:RR}

Let $M$ be a compact $K$-manifold equipped with an invariant almost complex structure $J$. Let $p:\T M\to M$ be the projection. The complex vector bundle $(\T^* M)^{0,1}$ is $K$-equivariantly identified with the tangent bundle $\T M$ equipped with the complex structure $J$.
Let $h$ be the  Hermitian structure on  $(\T M,J)$ defined by : $h(v,w)=\Omega(v,J w) - i \Omega(v,w)$ for $v,w\in \T M$. The symbol 
$$
\Thom(M,J)\in 
\Gamma\left(\T M,\hom(p^{*}(\wedge_{\Cbb}^{even} \T M),\,p^{*}
(\wedge_{\Cbb}^{odd} \T M))\right)
$$  
at $(m,v)\in \T M$ is equal to the Clifford map
\begin{equation}\label{eq.thom.complex}
 \clif_{m}(v)\ :\ \wedge_{\Cbb}^{even} \T_m M
\longrightarrow \wedge_{\Cbb}^{odd} \T_m M,
\end{equation}
where $\clif_{m}(v).w= v\wedge w - \iota(v)w$ for $w\in 
\wedge_{\Cbb}^{\bullet} \T_{m}M$. Here $\iota(v):\wedge_{\Cbb}^{\bullet} 
\T_{m}M\to\wedge_{\Cbb}^{\bullet -1} \T_{m}M$ denotes the 
contraction map relative to $h$. Since $\clif_{m}(v)^2=-\| v\|^2 {\rm Id}$, the map  
$\clif_{m}(v)$ is invertible for all $v\neq 0$. Hence the characteristic set 
of $\Thom(M,J)$ corresponds to the $0$-section of $\T M$. 

\begin{defi}
To any $K$-equivariant complex vector bundle $E\to M$, we associate its Riemann-Roch character 
$$
\RR^K(M,E):=\indice^K_M(\Thom(M,J)\otimes E)\in R(K).
$$
\end{defi}

\begin{rem}
The character $\RR^{K}(M, E)$ is equal to the equivariant index of the Dolbeault-Dirac operator 
$\Dcal_E:=\sqrt{2}(\overline{\partial}_{E} + \overline{\partial}^*_{E})$, since $\Thom(M,J)\otimes E$  corresponds to the principal symbol  of $\Dcal_E$ (see \cite{B-G-V}[Proposition 3.67]).  
\end{rem}

\subsection{Definition of $\Qcal^{\Phi}$}\label{sec:Qcal-Phi}

Let $(M,\Omega_M,\Phi_M^K)$ be a compact Hamiltonian $K$-manifold pre-quantized by an 
equivariant line bundle $L_M$. Let $J$ be an invariant almost complex structure 
compatible with $\Omega$. Let $\RR^K(M,-)$ be the corresponding Riemann-Roch character. The topological index 
of $\Thom(M,J)\otimes L_M\in\Ko_{K}(\T M)$ is equal 
to the analytical index of the Dolbeault-Dirac operator $\sqrt{2}(\overline{\partial}_{L_M}+\overline{\partial}_{L_M}^*)$~:
\begin{equation}\label{eq:index-analytique-topologique}
    \Qcal_{K}(M)=\RR^K(M,L_M).
\end{equation}

\medskip


When $M$ is not compact the topological index of $\Thom(M,J)\otimes L_M$
is not defined. In order to extend the notion of geometric quantization to this 
setting we deform the symbol $\Thom(M,J)\otimes L$ in the ``Witten'' way \cite{pep-RR,pep-ENS,Ma-Zhang}. Consider the identification
$\xi\mapsto\wtde{\xi},\kgot^*\to\kgot$ defined by a $K$-invariant scalar product
on  $\kgot^*$. We define the {\em Kirwan vector field} on
$M$ : 
\begin{equation}\label{eq-kappa}
    \kappa_m= \left(\wtde{\Phi_M^K(m)}\right)_M(m), \quad m\in M.
\end{equation}

\begin{defi}\label{def:pushed-sigma}
The symbol  $\Thom(M,J)\otimes L$ pushed by the vector field $\kappa$ is the symbol $\clif^\kappa$ 
defined by the relation
$$
\clif^\kappa\vert_m(v)=\Thom(M,J)\otimes L\vert_m(v-\kappa_m)
$$
for any $(m,v)\in\T M$. More generally, if $E\to M$ is an equivariant complex vector bundle, one defines the symbol 
$\clif^\kappa_E$ with the same relation (with $E$ at the place of $L$).
\end{defi}

Note that $\clif^\kappa\vert_m(v)$ is invertible except if
$v=\kappa_m$. If furthermore $v$ belongs to the subset $\T_K M$
of tangent vectors orthogonal to the $K$-orbits, then $v=0$ and
$\kappa_m=0$.  Indeed $\kappa_m$ is tangent to $K\cdot m$ while
$v$ is orthogonal.

Since $\kappa$ is the Hamiltonian vector field of the function
$\frac{-1}{2}\|\Phi_M^K\|^2$, the set of zeros of $\kappa$ coincides with the set 
of critical points of $\|\Phi_M^K\|^2$. Finally we have 
$$
\Char(\clif^\kappa)\cap \T_K M \simeq
\Cr(\|\Phi_M^K\|^2).
$$

In general $\Cr(\|\Phi_M^K\|^2)$ is not compact, so $\clif^\kappa$ does not define a transversally elliptic 
symbol on $M$. In order to define a kind of index of $\clif^\kappa$, we proceed as follows. 
For any invariant open relatively compact subset 
$U\subset M$  the set  $\Char(\clif^\kappa\vert_{U})\cap \T_K U \simeq \Cr(\|\Phi\|^2)\cap U$ is compact when  
\begin{equation}\label{eq:condition-U}
\partial U\cap \Cr(\|\Phi\|^2)=\emptyset.
\end{equation}
When (\ref{eq:condition-U}) holds we denote 
\begin{equation}\label{eq:Q-Phi-U}
\Qcal^{\Phi}_K(U):= \indice^K_{U}(\clif^\kappa\vert_{U})\quad
   \in\quad \Rforc(K)
\end{equation}
the equivariant index of the transversally elliptic symbol $\clif^\kappa\vert_{U}$.

Let us recall the description of the critical points of $\|\Phi_M^K\|^2$, when the moment map $\Phi_M^K$ is proper. We knows that 
$m\in \Cr(\|\Phi_M^K\|^2)$ if and only if $\wtde{\beta}_M(m)=0$ for $\beta=\Phi(m)$. Hence the set $\Cr(\|\Phi_M^K\|^2)$ has the following decomposition
\begin{equation}\label{eq:crit-Phi}
\Cr(\|\Phi_M^K\|^2)=\bigcup_{\beta\in \kgot^*} \ M^{\wtde{\beta}}\cap (\Phi_M^K)^{-1}(\beta)\nonumber\\
=\bigcup_{\beta\in\Bcal} \ \underbrace{K\cdot (M^{\wtde{\beta}}\cap (\Phi_M^K)^{-1}(\beta))}_{Z_\beta},
\end{equation}
where $\Bcal$ is a subset of the Weyl chamber $\tgot^*_+$. We denote by $B_r\subset \tgot^*$ the open 
ball $\{\xi\in\tgot^*\  | \ \|\xi\|<r\}$. The following Proposition is proved in \cite{pep-formal-2}.

\begin{prop}\label{prop:crit-Phi}
$\bullet$ For any $r>0$, the set $\Bcal\cap B_r$ is finite.

$\bullet$ The set of singular values of $\|\Phi_M^K\|^2: M\to\Rbb$ forms a sequence $0\leq r_1< r_2<\ldots < r_k<\ldots$  
which is finite  if and only if $\Cr(\|\Phi_M^K\|^2)$ is compact. In the other case $\lim_{k\to\infty}r_k=\infty$.
\end{prop}

\medskip

For any $\beta\in \Bcal$, we consider a relatively compact open invariant neighbourhood 
$\Ucal_\beta$ of  $Z_\beta$ such that $\Cr(\|\Phi_M^K\|^2)\cap \overline{\Ucal_\beta}= Z_\beta$.  The excision property tell us that 
the generalized character $\Qcal^{\Phi}_K(\Ucal_\beta)=\indice^{K}_{\Ucal_\beta}(\clif^\kappa\vert_{\Ucal_\beta})$ does not depend of the choice of $\Ucal_\beta$. In order to simplify the notations we consider the following 

\begin{defi}\label{def:Q-beta}
$\bullet$ We denote $\Qcal^{\beta}_K(M)\in\Rforc(K)$ the equivariant index\footnote{The index of $\clif^\kappa\vert_{\Ucal_\beta}$ was denoted $RR^{^K}_\beta(M,L)$ in \cite{pep-RR}}  of the transversally elliptic symbol $\clif^\kappa\vert_{\Ucal_\beta}$.

$\bullet$ When $E\to M$ is an equivariant complex vector bundle, we denote $\RR^{K}_\beta(M,E)$ the equivariant index of the transversally elliptic symbol $\clif^\kappa_E\vert_{\Ucal_\beta}$.
\end{defi}

The following crucial property is proved in \cite{Ma-Zhang,pep-formal-2}.

\begin{theo}\label{prop:Q-beta-support}
A representation $V_\lambda^K$ occurs in the generalized character $\Qcal^{\beta}_K(M)\in \Rfor(K)$ only if 
$\|\lambda\|\geq\|\beta\|$. 
\end{theo}

\medskip

\begin{defi}\label{def:Q-Phi}
The generalized character $\Qcal^{\Phi}_K(M)\in\Rfor(K)$ is defined by 
\begin{equation}\label{eq:Q-Q-beta}
\Qcal^{\Phi}_K(M)= \sum_{\beta\in\Bcal} \Qcal^{\beta}_K(M).
\end{equation}
\end{defi}

The sum (\ref{eq:Q-Q-beta}) converges in $\Rfor(K)$ since we know after Theorem \ref{prop:Q-beta-support} 
that the multiplicity of $V_\lambda^K$ in $\Qcal^{\beta}_K(M)$ is zero when $\|\beta\|>\|\lambda\|$. 

\medskip

We finish this section, by recalling a result that will be needed in Section \ref{sec:proof-theo-qfor-Y}. 
Suppose that  $\kgot=\kgot_1\oplus\kgot_2$ where $[\kgot_1,\kgot_2]=0$ and $\kgot_i$ are the Lie algebras of closed connected subgroups $K_i$. 
We assume that the moment map $\Phi_M^{K_1}: M\to \kgot^*_1$ relative to the $K_1$-action is {\em proper}.  Let us explain how we can use the $K$-invariant proper map 
$\|\Phi_M^{K_1}\|^2$ instead of $\|\Phi_M^K\|^2$ in order to defined the the geometric quantization $\Qcal^{\Phi}_K(M)$. 

Let us choose $\tgot=\tgot_1\oplus\tgot_2$ such that $\tgot_i\subset \kgot_i$ is a maximal abelian sub-algebras. We start a decomposition
\begin{equation}\label{eq:crit-Phi-agot}
\Cr(\|\Phi_M^{K_1}\|^2)=\bigcup_{\beta\in\Bcal_1} \ \underbrace{K\cdot (M^{\wtde{\beta}}\cap (\Phi_M^{K_1})^{-1}(\beta))}_{Z^1_\beta},
\end{equation} 
with $\Bcal_1\subset \tgot_1^*$.

Let $\kappa_1$ be the Hamiltonian vector field of $\frac{-1}{2}\|\Phi_M^{K_1}\|^2$, and let $\clif^{\kappa_1}$ be the corresponding pushed symbol. For any 
$\beta\in \Bcal_1$, we consider a relatively compact open $K$-invariant neighbourhood 
$\Ucal_\beta^1$ of  $Z_\beta^1$ such that $\Cr(\|\Phi_M^{K_1}\|^2)\cap \overline{\Ucal_\beta^1}= Z_\beta^1$.  
 We denote $\Qcal^{\beta,1}_K(M)\in\Rfor(K)$ the equivariant index of the $K_1$-transversally elliptic symbol $\clif^{\kappa_1}\vert_{\Ucal_\beta^1}$. 
 Theorem \ref{prop:Q-beta-support} admits the following extension 
 \begin{theo}\label{prop:Q-beta-support-agot}
A representation $V_\lambda^K$ occurs in the generalized character $\Qcal^{\beta,1}_K(M)$ only if 
$\|\lambda_1\|\geq\|\beta\|$. Here $\lambda\in\wedge^*\subset\tgot^*$ is decomposed in $\lambda=\lambda_1\oplus \lambda_2$ with 
$\lambda_i\in\tgot^*_i$.
\end{theo}

Like in Definition \ref{def:Q-Phi}, we can define the generalized character $\Qcal^{\Phi_1}_K(M)\in\Rfor(K)$  by 
\begin{equation}\label{eq:Q-Q-beta-agot}
\Qcal^{\Phi_1}_K(M)= \sum_{\beta\in\Bcal_1} \Qcal^{\beta,1}_K(M).
\end{equation}

In \cite{pep-formal-2}[Section 4.1],  we prove the following 
\begin{theo}\label{theo:quantization-phi-1}
Let $(M,\Omega_M,\Phi^K_M)$ be a proper Hamiltonian $K$-manifold that is pre-quantized. If the moment map 
$\Phi_M^{K_1}: M\to \kgot^*_1$ is proper, we have 
$$
\Qcal^{\Phi}_K(M)=\Qcal^{\Phi_1}_K(M)
$$
in $\Rfor(K)$.
\end{theo}

\subsection{Proof of Theorem \ref{theo:formal-G-K} under Assumption {\bf A1}}\label{sec:prof-theorem-formal-G-K-1}

In this section we consider the manifold $M=G\times_K Y$, where $(Y,\Omega_Y,\Phi^K_Y)$ is a Hamiltonian 
$K$-manifold pre-quantized by a line bundle $L_Y$. We suppose that the moment map $\Phi^K_Y$ is proper, and that the Kirwan polytope $\Delta_K(Y)$ is contained in the cone $\chol(z)\subset \tgot_{se}^*$. 

Then on $M$, we have an induced $G$-invariant symplectic form $\Omega_M$ and a  moment 
map $\Phi^G_M: M\to \ggot^*$ defined by $\Phi^G_M([g,y])=g\cdot\Phi^K_M(y)$. We know that 
the line bundle $L_M=(G\times L_Y)/K$ pre-quantizes the Hamiltonian manifold $(M,\Omega_M,\Phi^G_M)$.
Let us consider the $K$-action on $M$: the moment map $\Phi^K_M$ is also proper.

We are then in a setting where the formal geometric quantization of $M$ and $Y$ relatively to the 
$K$-action are well defined: $\Qcal^\Phi_K(M),\,\Qcal^\Phi_K(Y)\in\Rfor(K)$. 
The aim of this section is to prove that 
\begin{equation}
\Qcal^\Phi_K(M)=\Qcal^\Phi_K(Y)\otimes S^\bullet(\pgot),
\end{equation}
when $(M,\Omega_M,\Phi^G_M)$ satisfies Assumption {\bf A1}. Then the set (see Theorem \ref{theo:properGK}) 
$$
\Cr(\|\Phi^G_M\|^2)=\Cr(\|\Phi^K_M\|^2)=\Cr(\|\Phi^K_Y\|^2)=\bigcup_{\beta\in\Bcal} \underbrace{K\cdot\left(Y^{\tilde{\beta}}\cap(\Phi^K_Y)^{-1}(\beta)\right)}_{Z_\beta}
$$
is compact: the parametrizing set $\Bcal$ is {\em finite}. So we have 
$\Qcal^\Phi_K(M)=\sum_{\beta\in\Bcal}\Qcal^\beta_K(M)$ and $\Qcal^\Phi_K(Y)=\sum_{\beta\in\Bcal}\Qcal^\beta_K(Y)$, and we reduce to show the following 
\begin{theo}
For any $\beta\in\Bcal$, the following relation 
\begin{equation}
\Qcal^\beta_K(M)=\Qcal^\beta_K(Y)\otimes S^\bullet(\pgot),
\end{equation}
holds in $\Rfor(K)$.
\end{theo}

\begin{proof} Let $\kappa_M$ be the Kirwan vector field on $M$ associated to the moment map 
$\Phi^K_M$. Let $J_M$ a $K$-invariant almost complex structure compatible with $\Omega_M$, and 
let $\Ucal_\beta\subset M$ be a (small) neighbourhood of $Z_\beta$ in $M$.

The symbol  $\Thom(M,J_M)\otimes L_M$ pushed by the vector field $\kappa_M$ is denoted  $\clif^\kappa_M$.
By definition $\Qcal^\beta_K(M)$ is the equivariant index of the $K$-transversally elliptic symbol $\clif^\kappa_M\vert_{\Ucal_\beta}$. 
Note that $\Qcal^\beta_K(M)$ does not depend on the choice of the neighbourhood $\Ucal_\beta$ nor on the choice of the almost complex structure on $\Ucal_\beta\subset M$.

We use the $K$-diffeomorphism $\varphi:\pgot\times Y\simeq M$ defined by $\varphi(X,y)=[e^X,y]$. The Kirwan vector field 
$\kappa_{\pgot\times Y}:=\varphi^*(\kappa_M)$ is defined by the relations~: $\kappa_{\pgot\times Y}(X,y)=(\kappa_1(X,y),\kappa_2(X,y))\in \T_\pgot\times Y$ 
where\footnote{$[Z]_\kgot$ and $[X]_\pgot$ are respectively the $\kgot$ and $\pgot$ components of $Z\in\ggot$.}
$$
\kappa_1(X,y)= A_Y(y),\quad \kappa_2(X,y)=-[A,X]\quad \mathrm{and}\quad  A=[e^X\cdot\widetilde{\Phi^K_Y(y)}]_\kgot.
$$

The Kostant-Souriau line bundle $\varphi^*(L_M)$ is $K$-diffeomorphic with $L_Y$ since $Y$ is a deformation retract of $\pgot\times Y$. 
Let us compute the pull-back of the symplectic form $\Omega_{\pgot\times Y}=\varphi^*(\Omega_M)$ at $(0,y)$. For 
$v,v'\in\T_y Y$ and $\eta,\eta'\in \T_0\pgot=\pgot$, we have
\begin{eqnarray*}
\Omega_{\pgot\times Y}(\eta\oplus v, \eta'\oplus v')&=&\Omega_M(v\oplus\eta\cdot y, v'\oplus\eta'\cdot y)\\
&=&\Omega_Y(v,v') +\langle\Phi^K_Y(y),[\eta,\eta']\rangle.
\end{eqnarray*}

\begin{lem}\label{lem:negatif}
$\langle\xi,[\eta,\mathrm{ad}(z)\eta]\rangle=-([\tilde{\xi},\eta], [z,\eta])<0$ for any $\xi\in K\cdot\chol(z)$ and any $\eta\in\pgot\setminus\{0\}$. 
\end{lem}
\begin{proof} Recall that the scalar product on $\ggot$ is defined by $(X,Y)=-b(X,\Theta(Y))$. Hence 
\begin{eqnarray*}
\langle\xi,[\eta,\mathrm{ad}(z)\eta]\rangle&=&-b(\tilde{\xi},\Theta([\eta,\mathrm{ad}(z)\eta]))\\
&=&(\mathrm{ad}(z)\mathrm{ad}(\tilde{\xi})\eta,\eta)\\
&=&(\mathrm{ad}(z)\mathrm{ad}(\tilde{\xi'})\eta',\eta')
\end{eqnarray*}
where $\xi=k\cdot\xi'$ with $\xi'\in\chol(z)$ and $\eta=k\cdot\eta'$ for some $k\in K$. We can then check that 
the symmetric endomorphism $\mathrm{ad}(z)\mathrm{ad}(\tilde{\xi'}):\pgot\to\pgot$ is negative definite when 
$\xi'\in\chol(z)$: the Lemma is proved.
\end{proof}

\medskip

If $J_Y$ is a $K$-invariant almost complex structure on $Y$ compatible with $\Omega_Y$, the last Lemma tells us that 
$(-\mathrm{ad}(z),,J_Y)$ is a $K$-invariant almost complex structure on $\pgot\times Y$ compatible with 
$\Omega_{\pgot\times Y}$ in a neighbourhood of $Y$. 

Let us fix $\Ucal_\beta$, such that $\varphi^{-1}(\Ucal_\beta)=B_r \times \Vcal_\beta$ where $\Vcal_\beta$ is a neighbourhood of $Z_\beta$ in $Y$ and $B_r:=\{X\in\pgot\,\vert\, \|X\|<r\}$. The almost complex structure $J_M$
on $\Ucal_\beta$ defined by $\varphi^*(J_M)= (-\mathrm{ad}(z),,J_Y)$ is compatible with $\Omega_M$ if  $\Vcal_\beta$ and $B_r$ are small enough. 
Finally we see that the symbol $\varphi^*(\clif^\kappa_M\vert_{\Ucal_\beta})$ is equal to the product $\sigma_1\odot\sigma_2\vert_{B_r \times \Vcal_\beta}$, where 
$$
\sigma_2(X,y;\eta,v)=\clif(v-\kappa_1(X,y)),\quad (X,y;\eta, v)\in\T(\pgot\times Y),
$$ 
acts on $\wedge^*_\Cbb\T_yY\otimes L_Y$, and 
$$
\sigma_1(X,y;\eta,v)=\clif(\eta-\kappa_2(X,y)),\quad (X,y;\eta, v)\in\T(\pgot\times Y),
$$ 
acts on $\wedge^*_\Cbb\pgot^{-}$ (here $\pgot^{-}$ denotes the complex $K$-module $(\pgot,-\mathrm{ad}(z))$).

\medskip

Let $\kappa_Y$ be the Kirwan vector field on $Y$ associated to the moment map $\Phi^K_Y$. We denoted $\clif^\kappa_Y$, 
the symbol  $\Thom(Y,J_Y)\otimes L_Y$ pushed by the vector field $\kappa_Y$. By definition $\Qcal^\beta_K(Y)$ is the 
equivariant index of the $K$-transversally elliptic symbol $\clif^\kappa_Y\vert_{\Vcal_\beta}$. 

The Atiyah symbol $\mathrm{At}_\pgot$ on $\pgot$ is defined by the following relations : for $(X,\eta)\in\T \pgot$,
\begin{equation}\label{eq:atiyah-symbol}
\mathrm{At}_\pgot(X,\eta):=\clif(\eta+[z,X]): \wedge^{\mathrm{\small even}}_\Cbb\pgot^{-}\longrightarrow \wedge^{\mathrm{\small odd}}_\Cbb\pgot^{-}.
\end{equation}

\begin{lem}The symbols $\sigma_1\odot\sigma_2\vert_{B_r \times \Vcal_\beta}$ and 
$\mathrm{At}_\pgot\odot \clif^\kappa_Y\vert_{B_r \times \Vcal_\beta}$ 
define the same class in $\Ko_K(\T_K (B_r \times \Vcal_\beta))$.
\end{lem}

\begin{proof}
We consider the paths $s\in [0,1]\mapsto A^s:=[e^{sX}\cdot\widetilde{\Phi^K_Y(y)}]_\kgot$, 
$\kappa^s_1(X,y)= A^s_Y(y)$, and $\kappa^s_2(X,y)=-[A^s,X]$. We define then the paths at the level of symbols : 
$\sigma_1^s$ and $\sigma^s_2$. We check that 
$$
\Char(\sigma_1^s\odot\sigma_2^s)\cap \T_K(Y\times\pgot)= \{(X,y; v,\eta)\,\vert\, v=A^s_Y(y)=0, \mathrm{and}\ \eta=[A^s,X]=0\}.
$$
But since $\Phi^K_Y(y)\in\kgot^*_{se}$, the condition $[A^s,X]=[e^X\cdot\widetilde{\Phi^K_Y(y)},X]_\pgot=0$ forces $X$ to be equal to $0$. Hence we get 
$$
\Char(\sigma_1^s\odot\sigma_2^s)\cap \T_K(Y\times\pgot)\simeq \Cr(\|\Phi^K_Y\|^2)\times\{0\},\quad \forall s\in[0,1].
$$
We have proved that $s\in [0,1]\mapsto \sigma_1^s\odot\sigma_2^s\vert_{B_r \times \Vcal_\beta}$ is an homotopy of transversally elliptic symbols:
 $\sigma_1\odot\sigma_2$ and $\sigma_1^0\odot\sigma_2^0$ define the same class in $\Ko_K(\T_K (B_r \times \Vcal_\beta))$.
 
 We see that $\sigma_2^0=\clif^\kappa_Y$ and we have
 $$
 \sigma_1^0(X,y;\eta,v)=\clif(\eta+[\widetilde{\Phi^K_Y(y)},X]).
 $$
 We consider another path of symbols 
 $$
 \tau^t(X,y;\eta,v)=\clif(\eta+[t\widetilde{\Phi^K_Y(y)}+(1-t)z,X]), \quad t\in [0,1].
 $$
We check that if $(X,y; \eta,v)\in \Char(\tau^t\odot\clif^\kappa_Y)\cap \T_K(\pgot\times Y)$ then the vector 
$\eta\oplus v\in \T_{(X,y)}(\pgot\times Y)$ is orthogonal to the vector field generated by $\widetilde{\Phi^K_Y(y)}$ and 
we have moreover $v=\kappa_Y(y)$ and $\eta=-[t\widetilde{\Phi^K_Y(y)}+(1-t)z,X]$. Thus 
\begin{eqnarray*}
0&=&\|\kappa_Y(y)\|^2+\left([t\widetilde{\xi}+(1-t)z,X], [\tilde{\xi},X]\right)\\
&=&\|\kappa_Y(y)\|^2+\underbrace{t\, \|[\tilde{\xi},X]\|^2 + (1-t) \left([z,X], [\tilde{\xi},X]\right)}_{\delta}
\end{eqnarray*}
where $\xi=\Phi^K_Y(y)\in K\cdot\chol(z)$. Since $\xi$ is strongly elliptic and thanks to Lemma \ref{lem:negatif}, we know that the term $\delta$ is strictly positive if $X\neq 0$, thus $\kappa_Y(y)=0$ and $X=0$.

We have proved that $t\in [0,1]\mapsto \tau^t\odot\clif^\kappa_Y\vert_{B_r \times \Vcal_\beta}$ is an homotopy of $K$-transversally elliptic symbols: $\sigma_1^0\odot\sigma_2^0$ and $\mathrm{At}_\pgot\odot \clif^\kappa_Y$ define the same class in 
$\Ko_K(\T_K (B_r \times \Vcal_\beta))$.
\end{proof}

At this stage we know that 
$$
\Qcal^\beta_K(M)=\indice^K_{B_r \times \Vcal_\beta}\left(\mathrm{At}_\pgot\vert_{B_r}\odot \clif^\kappa_Y\vert_{\Vcal_\beta}\right)\in \Rfor(K).
$$
Since $\clif^\kappa_Y\odot \mathrm{At}_\pgot$ is also $K$-transversally elliptic on $\Vcal_\beta\times \pgot$, the excision property gives also
$$
\Qcal^\beta_K(M)=\indice^K_{\Vcal_\beta\times \pgot}\left(\mathrm{At}_\pgot\odot \clif^\kappa_Y\vert_{\Vcal_\beta}\right)\in \Rfor(K).
$$

Let $S^1$ be the circle subgroup of $K$ with Lie algebra equal to $\Rbb z$. We can consider $\pgot$ as a $S^1\times K$-manifold. We note that the Atiyah symbol $\mathrm{At}_\pgot$ is $S^1\times K$-equivariant and $S^1$-transversally elliptic. Its index is computed in \cite{Atiyah74}, see also \cite{pep-RR}[Section 5]. We have the following relation 
$$
\indice^{S^1\times K}_{\pgot}(\mathrm{At}_\pgot)=S^\bullet(\pgot)
$$
in $\Rfor(S^1\times K)$.

So we have two classes $\mathrm{At}_\pgot\in \Ko_{S^1\times K}(\T_{S^1}\pgot)$, and $\clif^\kappa_Y\vert_{\Vcal_\beta}\in \Ko_{K}(\T_{K}\Vcal_\beta)$. By the {\em multiplicative property} (see Theorem \ref{theo:multiplicative-property}) we know that their product $\mathrm{At}_\pgot\odot \clif^\kappa_Y\vert_{\Vcal_\beta}\in \Ko_{K\times S^1}(\T_{K\times S^1}(Y\times\pgot))$ 
has the following $S^1\times K$-equivariant index
\begin{eqnarray*}
\indice^{S^1\times K}_{\pgot\times \Vcal_\beta}(\mathrm{At}_\pgot\odot\clif^\kappa_Y\vert_{\Vcal_\beta})
&=&\indice^{S_1\times K}_{\pgot}(\mathrm{At}_\pgot)\otimes\indice^{K}_{\Vcal_\beta}(\clif^\kappa_Y\vert_{\Vcal_\beta})\\
&=& S^\bullet(\pgot)\otimes\indice^{K}_{\Vcal_\beta}(\clif^\kappa_Y\vert_{\Vcal_\beta})\\
&=& S^\bullet(\pgot)\otimes\Qcal^\beta_K(Y)\quad \in\Rfor(S^1\times K).
\end{eqnarray*}

Finally, thanks to the {\em restriction property} (\ref{eq:1-2-index}), we know that 
$$
\Qcal^\beta_K(M)=\indice^K_{\pgot \times\Vcal_\beta}\left(\mathrm{At}_\pgot\odot \clif^\kappa_Y\vert_{\Vcal_\beta}\right)\in \Rfor(K)
$$
is equal to the restriction of 
$$
\indice^{S^1\times K}_{\pgot\times \Vcal_\beta}(\mathrm{At}_\pgot\odot\clif^\kappa_Y\vert_{\Vcal_\beta})=S^\bullet(\pgot)\otimes\Qcal^\beta_K(Y) \in \Rfor(S^1\times K)
$$
to the subgroup $K\croc S^1\times K$. The Theorem is then proved. 
\end{proof}

\begin{rem}
The Assumption {\bf A1} is used because we don't how to prove the equality
$$
\sum_{\beta\in\Bcal}\left(\Qcal_K^\beta(Y)\otimes S^\bullet(\pgot)\right)= 
\left(\sum_{\beta\in\Bcal}\Qcal_K^\beta(Y)\right)\otimes S^\bullet(\pgot)
$$
when the set $\Bcal$ is not finite, e.g. the set $\Cr(\|\Phi^K_M\|^2)$ is non-compact.
\end{rem}

\subsection{Proof of Theorem \ref{theo:qfor-Y}}\label{sec:proof-theo-qfor-Y}

Here we work with a pre-quantized Hamiltonian $K$-manifold $(P,\Omega_P,\Phi^K_P)$, and we assume that 
the map $\langle\Phi^K_P,z\rangle$ is proper. Here $\Rbb z$ is the Lie algebra of a circle subgroup $S^1\subset K$ 
contained in the center of $K$. 

We are in the context of Theorem \ref{theo:quantization-phi-1}. We have a decomposition $\kgot=\kgot_1\oplus \kgot_2$ 
where $\kgot_1:=\Rbb z$ and $\kgot_2$ are ideals of $\kgot$ and the moment map $\langle\Phi^K_P,z\rangle$ relative to the 
$S^1$-action is proper. Then we have the following equality 
\begin{equation}\label{eq:proof-Y-1}
\qfor_K(M)=\Qcal^\Phi_K(P)=\Qcal^{\langle\Phi,z\rangle}_K(P)\quad \in \Rfor(K)
\end{equation}
where the right hand side is computed via a localization procedure on the set $\Cr(\varphi_P)$ of critical points of 
the proper map $\varphi_P:=(\langle\Phi^K_P,z\rangle)^2$.  We note that 
$$
\Cr(\varphi_P)= \varphi_P^{-1}(0)\bigcup P^z.
$$

We are interested in the following cases
\begin{enumerate}
\item $P$ is a proper Hamiltonian $G$-manifold $(M,\Omega_M,\Phi^G_M)$ with a moment map taking values in $G\cdot\chol(z)$, and which satisfies Assumption {\bf A2}.
\item $P$ is the symplectic slice $Y$ of the former case $M:=G\times_K Y$.
\end{enumerate} 

Thanks to Lemma \ref{lem-phi-z-positif}, we know that in the two cases described above, the proper map $\varphi_P$ is strictly positive : hence 
$\varphi_P^{-1}(0)=\emptyset$.  Let us compute the generalized character $\Qcal^{\langle\Phi,z\rangle}_K(P)$ in this case. 

Let $\kappa_\varphi$ be the Hamiltonian vector field of $\frac{-1}{2}\varphi_P$. The symbol  $\Thom(P,J_P)\otimes L_P$ pushed by the vector field $\kappa_\varphi$ is denoted  $\clif^{\varphi}_P$. Let $\Bcal_P$ the set of connected component of $P^z$.  For any $\Xcal\in \Bcal_P$, we consider a relatively compact open $K$-invariant neighbourhood $\Ucal_\Xcal$ of  $\Xcal$ such that $\Cr(\varphi_P)\cap \overline{\Ucal_\Xcal}= \Xcal$.  We denote $\Qcal^{\Xcal}_K(P)\in\Rfor(K)$ the equivariant index of the $S^1$-transversally elliptic symbol $\clif^{\varphi}\vert_{\Ucal_\Xcal}$. 

When $\varphi_P^{-1}(0)=\emptyset$, the generalized character $\Qcal^{\langle\Phi,z\rangle}_K(P)$ is defined by the relation 
\begin{equation}\label{eq:proof-Y-2}
\Qcal^{\langle\Phi,z\rangle}_K(P)=\sum_{\Xcal\in\Bcal_P}\Qcal^{\Xcal}_K(P)\in\Rfor(K).
\end{equation}

For $\Xcal\in\Bcal_P$,  we denote 
\begin{itemize}
\item $L_\Xcal$ the restriction of the Kostant-Souriau line bundle $L_P$ on $\Xcal$,
\item $\Ncal_{\Xcal}$ the normal bundle of $\Xcal$ in $P$, and $\vert\Ncal_{\Xcal}\vert^{z},\Ncal_{\Xcal}^{+,z}$ 
are the $z$-polarized versions (see Section \ref{subsec:quantization-Y}).
\end{itemize}

If we use (\ref{eq:proof-Y-1}) and (\ref{eq:proof-Y-2}), the proof of Theorem \ref{theo:qfor-Y} is reduced to the following
\begin{prop}
We have the following equality in $\Rfor(K)$:
\begin{equation}\label{eq:proof-Y-3}
\Qcal^{\Xcal}_K(P)=(-1)^{r_\Xcal}\RR^{K}\left(\Xcal, L_{\Xcal}\otimes \det(\Ncal^{+,z}_{\Xcal})
\otimes S^\bullet(\vert\Ncal_{\Xcal}\vert^{z})\right),
\end{equation}
where $r_\Xcal$ is the complex rank of $\Ncal^{+,z}_{\Xcal}$.
\end{prop}
\begin{proof}Relations (\ref{eq:kostant=rel}) show that $\kappa_\varphi= \langle\Phi^K_P,z\rangle z_P$. Since 
$\langle\Phi^K_P,z\rangle>0$ in a neighbourhood of $\Ucal_\Xcal$, we can replace $\kappa_\varphi$ by the vector field $z_P$
without changing the index of the corresponding transversally elliptic operator. This means that $\Qcal^{\Xcal}_K(M)$ is equal to the index 
of $\sigma^z\vert_{\Ucal_\Xcal}$, where the symbol $\sigma^z$ is defined by : for $(m,v)\in\T P$,
\begin{equation}\label{eq:z-symbol}
\sigma^z(m,v):=\clif(v-z_P(m)): \wedge^{\mathrm{\small even}}_\Cbb\T_m P\otimes L_P\vert_m\longrightarrow \wedge^{\mathrm{\small odd}}_\Cbb\T_m P\otimes L_P\vert_m.
\end{equation}
We have proved in \cite{pep-RR}[Theorem 5.8], that the index of $\sigma^z\vert_{\Ucal_\Xcal}$ is equal to the right hand side of 
(\ref{eq:proof-Y-3}). Hence the proof is completed.
\end{proof}

\medskip

We want now to clarify the convergence of the sum that appears in (\ref{eq:proof-Y-2}), when $P^z$ is non-compact.

\medskip

Let $T$ be a maximal torus in $K$: it contains the circle subgroup $S^1$. Let $\wedge\subset \tgot$ be the lattice which is the kernel of $\exp:\tgot\to T$. Let $z_o\in \Rbb^{>0}z\cap \wedge$ that generates the sub-lattice $\Rbb z\cap \wedge$: the torus $S^1$ acts on an irreducible representation $V_\mu^K$ through the character $t\mapsto t^n$ with $n=\frac{\langle\mu,z_o\rangle}{2\pi}\in\Zbb$. We have then a graduation $R(K)=\sum_{n\in\Zbb}R_n(K)$ where $R_n(K)$ is the group generated by the representations $V_\mu^K$ such that 
$\frac{\langle\mu,z_o\rangle}{2\pi}=n$. We see that $R_n(K)\cdot R_m(K)\subset R_{n+m}(K)$.

For any $n\in\Zbb$, we denote $R_{\geq n}(K)$ (resp. $\Rfor_{\geq n}(K)$) the subgroup formed by the finite (resp. infinite) sum $\sum_{l\geq n} E_l$ where 
$E_l\in R_l(K)$. We have the following basic lemma
\begin{lem}\label{lem:Rfor-positive}
$\bullet$ If $A\in \Rfor_{\geq n}(K)$ and $B\in\Rfor_{\geq m}(K)$, then 
the product $A\cdot B$ is well-defined and belongs to $\Rfor_{\geq n +m}(K)$.

$\bullet$ An infinite sum $\sum_{n\geq 0} A_n$, with $A_n\in \Rfor_{\geq n}(K)$, converges in $\Rfor_{\geq 0}(K)$.
\end{lem}

\begin{proof}
The proof is left to the reader. 
\end{proof}

\medskip

For $\Xcal\in\Bcal_P$, the action of $S^1$ is trivial on $\Xcal$, and relation (\ref{eq:kostant-L}) shows that $S^1$ acts on the fibres of 
Kostant-Souriau line bundle $L_\Xcal$ through the character $t\mapsto t^{n(\Xcal)}$, where $n(\Xcal)=\frac{\langle\Phi^K_P(\Xcal),z_o\rangle}{2\pi}$ is a strictly positive integer.

\begin{prop}\label{prop:Rfor-positive}
$\bullet$ The generalized character $\Qcal^\Xcal_K(P)$ belongs to $\Rfor_{\geq n(\Xcal)}(K)$.

$\bullet$ The sum $\sum_{\Xcal\in\Bcal_P}\Qcal^{\Xcal}_K(P)$ converges in $\Rfor_{\geq 0}(K)$.
\end{prop}
\begin{proof}
The generalized character $\Qcal^\Xcal_K(P)$ is equal to the sum $(-1)^{r(\Xcal)}\sum_{p\geq 0} E_p$, with 
$$
E_p=\RR^{K}(\Xcal, L_{\Xcal}\otimes \det(\Ncal^{+,z}_{\Xcal})\otimes S^p(\vert\Ncal_{\Xcal}\vert^{z}))\in R(K).
$$
Since the group $S^1$ acts on the fibres of the polarized bundles $\Ncal^{+,z}_{\Xcal}$ and  $\vert\Ncal_{\Xcal}\vert^{z}$ through characters 
$t^n$ with $n>0$, we see that $E_p\in R_{\geq n(\Xcal)+p}(K)$. Hence $\Qcal^\Xcal_K(P)=(-1)^{r(\Xcal)}\sum_{p\geq 0} E_p$ converges 
in $\Rfor_{\geq n(\Xcal)}(K)$.

For the second point we see that $\sum_{\Xcal\in\Bcal_P}\Qcal^{\Xcal}_K(P)= \sum_{n\geq 0} A_n$ with 
$$
A_n=\sum_{n(\Xcal)=n}\Qcal^{\Xcal}_K(P)\in \Rfor_{\geq n}(K).
$$
The former sum is finite (and then well-defined) because the map $\langle\Phi^K_P,z_o\rangle$ is proper : for any $C>0$, we have 
only a finite  number of $\Xcal\in\Bcal_P$ such that $\langle\Phi^K_P(\Xcal),z_o\rangle\leq C$. The last point is proved.

\end{proof}

\subsection{Proof of Theorem \ref{theo:formal-G-K} under Assumption {\bf A2}}\label{sec:prof-theorem-formal-G-K-2}

The proof of Theorem \ref{theo:formal-G-K}, in the case where Assumption {\bf A2} is satisfied, follows directly of the results of 
the Section \ref{sec:proof-theo-qfor-Y} applied to the following two cases:
\begin{enumerate}
\item $(M,\Omega_M,\Phi^G_M)$ is a proper Hamiltonian $G$-manifold with a moment map taking values in $G\cdot\chol(z)$, and which satisfies Assumption {\bf A2}.
\item $Y$ is the symplectic slice of the former case $M:=G\times_K Y$.
\end{enumerate} 

Note that the fixed point set $M^z$ and $Y^z$ coincide. For a connected component $\Xcal$ of $Y^z$, let $\Ncal_\Xcal$ (resp. $\Ncal'_\Xcal$) 
be the normal bundle of $\Xcal$ in $M$ (resp. $Y$). Since the normal bundle of $Y$ in $M$ is the trivial bundle $Y\times \pgot$, we have $\Ncal_\Xcal=\Ncal'_\Xcal\oplus \pgot$. A small computation shows that 
$$
\Ncal_\Xcal^{+,z}=(\Ncal'_\Xcal)^{+,z}\quad \mathrm{and} \quad
\vert\Ncal_\Xcal\vert^{+,z}=\vert\Ncal'_\Xcal\vert^{+,z}\oplus (\pgot,\mathrm{ad}(z)).
$$

Finally (\ref{eq:proof-Y-1}) and Proposition \ref{eq:proof-Y-3} gives
\begin{eqnarray*}
\qfor_K(M)&=&
\sum_{\Xcal}(-1)^{r_\Xcal}\RR^{K}\left(\Xcal, L_{\Xcal}\otimes \det(\Ncal^{+,z}_{\Xcal})\otimes S^\bullet(\vert\Ncal_{\Xcal}\vert^{z})\right)\\
&=&\sum_{\Xcal}(-1)^{r_\Xcal}\RR^{K}\left(\Xcal, L_{\Xcal}\otimes \det(\Ncal'_\Xcal)^{+,z}\otimes S^\bullet(\vert\Ncal'_{\Xcal}\vert^{z})\otimes 
S^\bullet(\pgot)\right)\\
&=&\left(\sum_{\Xcal}(-1)^{r_\Xcal}\RR^{K}\left(\Xcal, L_{\Xcal}\otimes \det(\Ncal'_\Xcal)^{+,z}\otimes S^\bullet(\vert\Ncal'_{\Xcal}\vert^{z})\right)\right)\otimes 
S^\bullet(\pgot)\\
&=& \qfor_K(Y)\otimes S^\bullet(\pgot).
\end{eqnarray*}
We know that the term $\sum_{\Xcal}(-1)^{r_\Xcal}\RR^{K}\left(\Xcal, L_{\Xcal}\otimes \det(\Ncal'_\Xcal)^{+,z}\otimes S^\bullet(\vert\Ncal'_{\Xcal}\vert^{z})\right)$ belongs to $\Rfor_{\geq 0}(K)$ (see Proposition \ref{prop:Rfor-positive}).  We see also that $S^\bullet(\pgot)\in \Rfor_{\geq 0}(K)$. Hence 
their product is well-defined (see Lemma \ref{lem:Rfor-positive}).

\section{Appendix: proof of Lemma \ref{lem:restriction-admissible}}
 
Let $m=\sum_{\mu\in \Ghol(z)}m_\mu V^G_\mu \in \Rfor(G,z)$ that is $G'$-admissible : the map $\pi_{\ggot',\ggot}: G\cdot\mathrm{Support}(m)\to (\ggot')^*$ is proper. We start with the 

\begin{lem} $\bullet$ There exists a closed subset $C_m\subset\cholp(z)$ such that \break 
$\pi_{\ggot',\ggot}\left(G\cdot\mathrm{Support}(m)\right)= G'\cdot C_m$.

$\bullet$ The projection $\pi_{\kgot',\ggot'}: G'\cdot C_m \to (\kgot')^*$ is proper.

$\bullet$ The projection $\pi_{\kgot',\kgot} : \pi_{\kgot,\ggot}(G\cdot\mathrm{Support}(m))\to (\kgot')^*$ is proper.
\end{lem}

\begin{proof}
First we note that $G\cdot\mathrm{Support}(m)$ is closed in $\ggot^*$. Thanks to point $(d)$ of Proposition \ref{prop:cone-projection}, we know that $\pi_{\ggot',\ggot}\left(G\cdot\mathrm{Support}(m)\right)= G'\cdot C_m$, with $C_m= \pi_{\ggot',\ggot}(G\cdot\mathrm{Support}(m))\cap (\tgot')^*\subset \cholp(z)$. The properness assumption forces $G'\cdot C_m$ (resp. $C_m$) to be closed in $(\ggot')^*$ (resp. $\cholp(z)$). The first point is proved. 

Let $B_R:=\{ \xi' \in (\tgot')^* \, \vert \, \|\xi'\|\leq R \}$, and consider $e^{X'}\cdot\lambda'\in G'\cdot C_m\cap\pi_{\kgot',\ggot'}^{-1}(B_R)$.  
Since $\|\pi_{\kgot',\ggot'}(e^{X'}\cdot\lambda')\geq \|\lambda'\|$, we see that $\lambda'$ belongs to a compact subset $\Kcal$ of $\cholp(z)$. 
Then there exists a constant $c(\Kcal)>0$ such that $\|\pi_{\kgot',\ggot'}(e^{Y}\cdot\lambda')\|\geq c(\Kcal)\|X'\|^2$ for any $Y\in\pgot'$ and 
$\lambda'\in\Kcal$ (see the proof of Proposition \ref{theo:properGK}). Finally, we have proved that $e^{X'}\cdot\lambda'$ belongs to a compact subset of $G'\cdot C_m$.

Let $\eta\in G\cdot\mathrm{Support}(m)$ such that $\xi=\pi_{\kgot,\ggot}(\eta)\in \pi_{\kgot',\kgot}^{-1}(B_R)$. Then 
$\pi_{\kgot',\kgot}(\xi)=\pi_{\kgot',\ggot}(\eta)=\pi_{\kgot',\ggot'}\circ\pi_{\ggot',\ggot}(\eta)\in B_R$. Thanks to the former point, we know that $\pi_{\ggot',\ggot}(\eta)$ is bounded. Since $\pi_{\ggot',\ggot}: G\cdot\mathrm{Support}(m)\to (\ggot')^*$ is proper, it implies that 
$\eta$ and $\xi=\pi_{\kgot,\ggot}(\eta)$ are bounded.

\end{proof}

\medskip

Consider now the restriction of $m\in\Rfor(G,z)$ to $K$ and $G'$.  The generalized character $n:=\mathbf{r}_{K,G}(m)=\sum_{\lambda\in\Khol(z)} n_\lambda V^K_\lambda$, is defined by the relations
$$
n_\lambda=\sum_{\mu\in\Ghol(z)} m_\mu \Qcal((G\cdot\mu)_{\lambda,K}),
$$
and $q:=\mathbf{r}_{G',G}(m)=\sum_{\nu\in\Gholp(z)} q_\nu V^{G'}_\nu$, is defined by the relations
$$
q_\nu=\sum_{\mu\in\Ghol(z)} m_\mu \Qcal((G\cdot\mu)_{\nu, G'}),
$$

We see that $\lambda\in \mathrm{Support}(n)$ only if there exists 
$\mu$ such that  $m_\mu \Qcal((G\cdot\mu)_{\lambda,K})\neq 0$:  hence 
$$
\mathrm{Support}(n)\subset \pi_{\kgot,\ggot}\left(G\cdot\mathrm{Support}(m)\right).
$$
Since $\pi_{\kgot',\kgot}$ is proper on $\pi_{\kgot,\ggot}\left(G\cdot\mathrm{Support}(m)\right)$, we have that 
$\pi_{\kgot',\kgot}$ is proper on $K\cdot \mathrm{Support}(n)$. Thus $n\in \Rfor(K,z)$ is $K'$-admissible. 

\medskip

We will now compare the following two elements of $\Rfor(K',z)$ : $A:=\mathbf{r}_{K',K}(n)=\sum_{\delta\in\Kholp(z)} a_\delta V^{K'}_\delta$
and $B:=\mathbf{r}_{K',G'}(q)= \sum_{\delta\in\Kholp(z)} b_\delta V^{K'}_\delta$. 

By definition, we have 
\begin{eqnarray*}
a_\delta&=&\sum_{\lambda\in\Khol(z)} n_\lambda \Qcal((K\cdot\lambda)_{\delta,K'})\\
&=&\sum_{\mu\in\Ghol(z)} m_\mu \underbrace{\left(\sum_{\lambda\in\Khol(z)}\Qcal((G\cdot\mu)_{\lambda,K})\Qcal((K\cdot\lambda)_{\delta,K'})\right)}_{X_{\mu,\delta}},
\end{eqnarray*}
and
\begin{eqnarray*}
b_\delta&=& \sum_{\nu\in\Gholp(z)} q_\nu \Qcal((G'\cdot\nu)_{\delta,K'})\\
&=&\sum_{\mu\in\Ghol(z)} m_\mu 
\underbrace{\left(\sum_{\nu\in\Ghol(z)} \Qcal((G'\cdot\nu)_{\delta,K'}) \Qcal((G\cdot\mu)_{\nu, G'})\right)}_{Y_{\mu,\delta}}.
\end{eqnarray*}

We have
\begin{eqnarray*}
X_{\mu,\delta}&=&\sum_{\lambda\in\Khol(z)}\Qcal((G\cdot\mu)_{\lambda,K})\Qcal((K\cdot\lambda)_{\delta,K'})\\
&=&\sum_{\lambda\in\Khol(z)}[V_\lambda^K : V_\mu^G\vert_K]\cdot [V_\delta^{K'} : V_\lambda^{K'}\vert_K]\\
&=&[V_\delta^{K'} :V_\mu^G\vert_{K'}]
\end{eqnarray*}
which is finite if $m_\mu\neq 0$. Similarly, we check that  $Y_{\mu,\delta}=[V_\delta^{K'} : V_\mu^G\vert_{K'}]$. 

Finally, we have proved that $\mathbf{r}_{K',K}\circ\mathbf{r}_{K,G}(m)=A=B=\mathbf{r}_{K',G'}\circ\mathbf{r}_{G',G}(m)$.


\end{document}